\documentclass[11pt]{amsart}
\usepackage{amsmath, amsthm, amssymb,amscd}
\usepackage{float}
\usepackage{graphicx}
\usepackage{xypic}
\usepackage[arrow, matrix, curve]{xy}
\usepackage{color}
\usepackage{enumitem}  
\usepackage{tikz-cd} 
\usepackage{url}

\usepackage{blindtext}

\usepackage[english]{babel}
\newcommand{\seq}{\subseteq}
\newcommand{\C}{\mathbb{C}}

\newcommand{\rank}{\rm{rank }}
\newtheorem{thm}{Theorem}[section]
\newtheorem*{thm-nl}{Theorem}
\newtheorem*{prop-nl}{Proposition}

\newtheorem{lem}[thm]{Lemma}

\def\PP{{\mathbf P}}

\def\Pic0{{\rm Pic}^0(X)}

\newtheorem*{cor-nl}{Corollary}
\newtheorem{conjecture}[thm]{Conjecture}
\newtheorem*{conjecture-nl}{Conjecture}
\newtheorem{defin}[thm]{Definition}
\newtheorem*{quest-nl}{Question}
\newtheorem*{quests-nl}{Questions}

\newtheorem{prop}[thm]{Proposition}

\theoremstyle{remark}
\newtheorem*{rem}{Remark}

\newtheorem{remark}[thm]{Remark}

\title{{The Rank of Syzygies of Canonical Curves}}
\date{\today}

\author[M. Kemeny]{Michael Kemeny}

\address{University of Wisconsin-Madison, Department of Mathematics, 480 Lincoln Dr
\hfill \newline\texttt{}
 \indent WI 53706, USA} \email{{\tt michael.kemeny@gmail.com}}

\begin{document}
\begin{abstract}
We prove that the linear syzygy spaces of a general canonical curve are spanned by syzygies of minimal rank.

\end{abstract}
\maketitle
\setcounter{section}{-1}
\section{Introduction}

The goal of this paper is to study the rank of the syzygies amongst the equations defining a canonical curve.\smallskip

Following Mumford's influential work \cite{mumford-quadratic}, one typically studies projective varieties $X \seq \PP_{\C}^n$ under the assumption that the varieties are \emph{defined by quadrics}, or, in in algebraic terms, the homogeneous ideal $I_X \seq \C[x_0,\ldots, x_n]$ is generated in degree two. This assumption is often satisfied in cases of interest. For example, this assumption is satisfied for canonical curves of genus $g$ and Clifford index at least two by the classical Petri Theorem \cite{arbarello-sernesi-petri}. \smallskip

Consider a variety $X \seq \PP_{\C}^n$ which is defined by quadrics. It is natural to ask what is the minimal $r$ one can find such that $X$ is defined by quadrics of rank $r$. For instance, Mumford proved in \cite[Thm.\ 1]{mumford-quadratic}, that if $i: X \hookrightarrow \PP^n$ is any projective variety of degree $d_0$, and if we compose $i$ with the $d^{th}$ Veronese embedding $v_d: \PP^n \hookrightarrow \PP^{\binom{n+d}{d}-1}$, then the new embedding
$$X \hookrightarrow \PP^{\binom{n+d}{d}-1}$$
is defined by quadrics of rank at most four for $d \geq d_0$. This result has been subsequently extended in interesting ways, see \cite{sidman-smith} and \cite{quadrics-rank3}.\smallskip

In the case of canonical curves of genus $g$ and Clifford index at least two, Andreotti--Mayer \cite{andreotti-mayer} and Arbarello--Harris \cite{arbarello-harris} proved that $C \seq \PP^{g-1}$ is defined by quadrics of rank at most four, provided the curve is general. This was extended to arbitrary curves by Green \cite{green-quadrics}.\smallskip

In foundational work \cite{green-koszul}, Green put the classical study of the projective geometry of varieties into a much wider context. Instead of merely studying the ideal $I_{X} \seq S:=\C[x_0,\ldots, x_n]$, Green had the insight that one should consider the entire \emph{minimal free resolution}
$$ \ldots F_2 \to F_1 \to I_X \to 0,$$
of the $S$-module $I$. Much of the \emph{intrinsic} geometry of the projective variety reveals itself through the \emph{extrinsic} invariants encoded in the resolution. The celebrated Green's Conjecture postulates that one can read the Clifford index of a curve $C$ off from the dimensions of the graded pieces of the modules $F_i$ appearing in the resolution of the ideal of a canonical curve $C \seq \PP^{g-1}$.\smallskip

Following Green's philosophy of generalizing from an ideal $I_X$ to the resolution $F_{\bullet} \to I_X$, Schreyer and von Bothmer defined a notion of rank that applies to linear generators of the modules $F_i$. Let $X$ be a projective variety, embedded via a very ample line bundle $L$, and assume that $(X,L)$ is \emph{normally generated}, \cite{mumford-quadratic}. We decompose the free modules $F_i$ into their graded pieces by writing
$$F_i=\bigoplus_{j\geq 1} \mathrm{K}_{i,j}(X,L) \otimes_{\C} S(-i-j),$$
where the vector space $\mathrm{K}_{i,j}(X,L)$, called the $(i,j)^{th}$ Koszul cohomology group. For any $\alpha \neq 0 \in \mathrm{K}_{p,1}(X,L)$, we have a well-defined notion of \emph{rank}, see Definition \ref{rank-defin}.\smallskip

The following conjecture of Schreyer simultaneously unifies and generalises both Voisin's Theorem on generic Green's Conjecture and the Andreotti--Mayer--Arbarello--Harris Theorem on the ideal sheaf of a canonical curve:
  \begin{conjecture}[The Geometric Syzygy Conjecture] \label{geo-conj}
 For a general curve of genus $g$, all linear syzygy spaces $\mathrm{K}_{p,1}(C,\omega_C)$ are spanned by syzygies of minimal rank $p+1$.
 \end{conjecture}
 Conjecture \ref{geo-conj} was proven for curves of genus $g \leq 8$ in \cite{bothmer-Transactions}.  We previously proved an important special case of Conjecture \ref{geo-conj}, namely the case when $g=2k$ has even genus and $p=k-1$ is the \emph{last} nonzero linear syzygy space, see \cite[Thm.\ 0.2]{kemeny-voisin}. \smallskip
 
  The goal of this paper is to prove Conjecture \ref{geo-conj} in full generality:
  \begin{thm}
  Conjecture \ref{geo-conj}  holds for a general curve of genus $g \geq 8$. 
\end{thm} \smallskip
 
 We always have the bound $\rm{rank}(\alpha) \geq p+1$ for $\alpha \in \mathrm{K}_{p,1}(X,L)$. Thus Conjecture \ref{geo-conj} amounts to the statement that the spaces $K_{p,1}(C,\omega_C)$ are generated by syzygies of the lowest possible rank. Syzygies of lowest rank have a geometric interpretation, \cite{bothmer-JPAA}, \cite{aprodu-nagel-nonvanishing}. Indeed, if $\rm{rank}(\alpha)=p+1$, then $X$ lies on a rational normal scroll $$X_{\alpha}  \seq \PP(H^0(L)^{\vee}),$$ called the \emph{syzygy scheme}, which is constructed from $\alpha$ in an explicit way, \cite{green-canonical}. The syzygy $\alpha$ then arises from a minimal rank syzygy at the end of the resolution of $X_{\alpha}$, and as such has an explicit description, see \cite{schreyer1} and \cite[\S 0.2]{lin-syz}. For instance, when $p=1$, syzygies of rank two provide linear determinantal equations for $C$ in the sense of \cite{sidman-smith}, and in particular are quadrics of rank four.\smallskip
 
 Similarly, syzygies of rank $p+2$ arise from linear sections of Grassmannians. Most general constructions of syzygies produce low rank syzygies. In particular, an influential construction of Green--Lazarsfeld \cite{green-koszul} produces syzygies $\alpha \in \mathrm{K}_{p,1}(X,L)$ of rank at most $p+2$ out of a decomposition $L=L_1 \otimes L_2$ with $h^0(L_i) \geq 2$ for $i=1,2$, \cite[\S 3.4.2]{aprodu-nagel}.\smallskip
 
  By applying the Green--Lazarsfeld construction to decompositions $\omega_C \simeq L_1 \otimes L_2$ of the canonical bundle of a curve, Green obtained his famous conjecture predicting the vanishing of linear syzygy spaces $K_{p,1}(C,\omega_C)$ in terms of the intrinsic geometry of a curve. Green's Conjecture has been proven for general curves by Voisin in \cite{V1}, \cite{V2}; see also \cite{AFPRW} and \cite{kemeny-voisin} for simpler proofs. \smallskip



 We end the introduction with a few words on the proof of Conjecture \ref{geo-conj}. Our basic strategy is the technique of \emph{projection of syzygies}, originally developed by Aprodu \cite[\S 2.2]{aprodu-nagel} and further refined in the paper \cite{projecting}. This method provides an inductive method to prove results on syzygies of a smooth curves $C$. One first identifies pairs of points to create a nodal curve $D$, for which the claim is covered by the base case of the induction. One then projects from each node and uses formulae describing how syzygies change under this operation. In our case the base case is provided by \cite{kemeny-voisin}. \smallskip
 
  In order to implement this strategy, several new tools are required. These tools are likely to be useful in other contexts. In particular, we develop a novel, \emph{explicit} formula for the projection map, see Proposition \ref{image-projection}. Moreover, we require an explicit formula for the construction of minimal rank syzygies out of line bundles, see Proposition \ref{one-dim-iso}. Our formula has the feature that it may be adapted to torsion free sheaves on singular curves which admit embeddings into a surface, see Section \ref{BN-sing-curves}. We then apply these formulae to nodal curves which arise as deformations of (singular) curves on K3 surfaces, see Section \ref{final-section}. A considerable amount of technical difficulty results from the fact that we may not assume that the minimal pencils on such curves come from line bundles, but are instead required to work with general torsion-free sheaves.\smallskip

\textbf{Acknowledgements} I thank Daniel Erman for comments on a draft.

\section{Rank of a Syzygy} \label{rank-section}
\subsection{Preliminaries} We begin by gathering a few preliminaries.
\begin{lem} \label{coker-bundle}
Let $X$ be a scheme which is locally of finite type over an algebraically closed field $k$. Let $\phi  : E \to F$ be a morphism of vector bundles. Assume that, for any closed point $x \in X$, $\phi(x): E \otimes k(x) \to F \otimes k(x)$ is injective. Then $G:=\mathrm{Coker}(\phi)$ is locally free.
\end{lem}
\begin{proof}
Without loss of generality we may assume that $X=\mathrm{Spec}(A)$ is affine. The dual map $\phi(x)^{\vee}$ is surjective for all $x \in X$. By Nakayama's Lemma, the morphism of sheaves $\phi^{\vee}$ is surjective. Consider $\widetilde{G}:=\mathrm{Ker}(\phi^{\vee})$ which fits into the short exact sequence
\begin{equation} \label{prelim-1}
0 \to \widetilde{G} \to F^{\vee} \to E^{\vee} \to 0.
\end{equation}
Then $\widetilde{G}$ is locally free, since $F^{\vee}$ and $E^{\vee}$ are, \cite[Ex.\ III.6.5]{hartshorne}. Dualizing (\ref{prelim-1}) gives the claim. 
\end{proof}

\subsection{Notation and Background on Syzygies} If $V$ is a vector space and $M$ is a graded $S_V:=\mathrm{Sym}(V)$ module, write $\mathrm{K}_{p,q}(M,V)$ for the middle cohomology of 
$$\bigwedge^{p+1}V \otimes M_{q-1} \to \bigwedge^{p}V \otimes M_{q} \to \bigwedge^{p-1}V \otimes M_{q+1}.$$ For any projective variety $X$, with line bundle $L$ and coherent sheaf $F$, we define a graded $S_L:=\mathrm{Sym}(\mathrm{H}^0(X,L))$ module
$$\Gamma_X(F,L):=\bigoplus_{q \in \mathbb{Z}} \mathrm{H}^0(X,L^{\otimes q} \otimes F).$$ 
Define $\mathrm{K}_{p,q}(X,F, L):=\mathrm{K}_{p,q}(\Gamma_X(F,L),\mathrm{H}^0(X,L))$. When $F=\mathcal{O}_X$ is trivial, we write $\Gamma_X(L)$ for $\Gamma_X(\mathcal{O}_X,L)$ and $\mathrm{K}_{p,q}(X, L)$ for $\mathrm{K}_{p,q}(X,\mathcal{O}_X, L)$. For any subspace $W \seq \mathrm{H}^0(X,L)$, we may consider $\Gamma_X(F,L)$ as a $S_W:=\mathrm{Sym}(W)$ module. Define $\mathrm{K}_{p,q}(X,F, L; W):=\mathrm{K}_{p,q}(\Gamma_X(F,L),W)$. 
If $F=\mathcal{O}_X$, we write $\mathrm{K}_{p,q}(X,L;W)$ for $\mathrm{K}_{p,q}(X,\mathcal{O}_X, L; W)$. We further define
$$b_{p,q}(X,F,L; W):=\dim \mathrm{K}_{p,q}(X,F, L; W).$$

The following proposition is a slight refinement of \cite[Prop.\ 1.29, Cor.\ 1.31]{aprodu-nagel}.
\begin{prop}[Semicontinuity of Koszul cohomology] \label{semi-cont}
	Suppose $\pi: \mathcal{X} \to S$ is a flat, projective morphism of finite-type schemes over $\C$ and assume $S$ is integral. Let $\mathcal{L}$, $\mathcal{F}$ be coherent sheaves on $\mathcal{X}$, flat over $S$, with $\mathcal{L}$ a line bundle. Let $\mathcal{Z} \seq \mathcal{X}$ be a closed subscheme, flat over $S$. Fix integers $p,q$. Assume
	$$ h^0(\mathcal{X}_S, \mathcal{L}_s \otimes I_{\mathcal{Z}_s}),\; h^0(\mathcal{X}_S, \mathcal{F}_s \otimes \mathcal{L}^{q-1}_s ),\; h^0(\mathcal{X}_S, \mathcal{F}_s \otimes \mathcal{L}^{q}_s ),\; h^0(\mathcal{X}_S, \mathcal{F}_s \otimes \mathcal{L}^{q+1}_s )$$
	are all constant for all closed points $s \in S$.
	 Then the function $\Psi :  S \to \mathbb{Z}$ with
	\begin{align*}
	\Psi(s) := b_{p,q}\left(\mathcal{X}_s,\mathcal{F}_s,\mathcal{L}_s; \mathrm{H}^0(\mathcal{X}_S, \mathcal{L}_s \otimes I_{\mathcal{Z}_s})\right)
	\end{align*}
	is upper semicontinuous.
	\end{prop}
\begin{proof}
	Without loss of generality, we may take $S=\mathrm{Spec}(R)$ to be affine. Set
	\begin{align*}
	\mathcal{A}:=\pi_* (\mathcal{L} \otimes I_{\mathcal{Z}}), \; \; \mathcal{B}_i := \pi_*(\mathcal{F} \otimes \mathcal{L}^{\otimes i}), \; i \in \{ q-1,q,q+1 \}.
	\end{align*}
	Then $\mathcal{A}, \mathcal{B}_{q-1}, \mathcal{B}_q, \mathcal{B}_{q+1}$ are all vector bundles on $S$ by Grauert's Theorem, and we have a complex
	{\small{$$\bigwedge^{p+1}\mathcal{A} \otimes \mathcal{B}_{q-1} \xrightarrow{\delta_1} \bigwedge^{p}\mathcal{A} \otimes \mathcal{B}_{q} \xrightarrow{\delta_2} \bigwedge^{p-1}\mathcal{A} \otimes \mathcal{B}_{q+1}$$}}
	which specializes to the Koszul complex on $\mathcal{X}_s$ for any closed point $s \in S$. We have
	{\small{\begin{align*}
	\Psi(s)&:= \dim \mathrm{Ker}(\delta_2\otimes k(s))-\dim \mathrm{Im}(\delta_1 \otimes k(s))\\
	&=\mathrm{rk}(\wedge^{p}\mathcal{A} \otimes \mathcal{B}_{q})-\dim \mathrm{Im}(\delta_2\otimes k(s))-\dim \mathrm{Im}(\delta_1 \otimes k(s)).
	\end{align*}}}
	It suffices now to show that, for a morphism $\Phi \; : \; \mathcal{O}_S^a \to \mathcal{O}_S^b$ of finitely-generated, free modules, the function
	$s \mapsto \mathrm{rk} (\Phi\otimes k(s))$ is lower semicontinuous. Indeed, the set $\{ s \in S \; : \; \mathrm{rk} (\Phi\otimes k(s)) \leq n \}$
		is closed with ideal defined by the entries of the matrix $\wedge^n \Phi$, for any integer $n$.
	\end{proof}

\subsection{Rank of a Syzygy} We call a pair $(X,L)$ a polarized variety if $X$ is an integral, projective variety and $L$ a very ample line bundle on $X$. We call $(X,L)$ \emph{normally generated} if the map $\mathrm{Sym}^n \mathrm{H}^0(X,L) \to \mathrm{H}^0(X,L^{\otimes n})$ is surjective for all $n \geq 0$.
\begin{prop} \label{inc-subspace}
	Let $(X,L)$ be normally generated. For any subspace $W \seq \mathrm{H}^0(X,L)$ the natural map $\mathrm{K}_{p,1}(X,L; W) \to \mathrm{K}_{p,1}(X,L)$ is injective for all $p \geq 0$.
	\end{prop}
\begin{proof}  
This is implicitly in \cite[\S 2]{aprodu-higher}, but we give a proof for completeness.\\ 
 
	Set $V:=\mathrm{H}^0(L)$ and $\PP:=\PP(\mathrm{H}^0(X,L))$. We have the exact Koszul complex
	{\small{$$\ldots \to \bigwedge^2 W \otimes S_V(-2) \to W \otimes S_V(-1) \to S_V \to S_{V/W} \to 0 $$}}
	of graded $S_V:=\mathrm{Sym}(V)$ modules, \cite[\S I.14]{peeva}. Twisting this complex appropriately, we see
	{\small{$$\wedge^{s+1}W \otimes \mathrm{Sym}^{t-1}(V) \to \wedge^{s}W \otimes \mathrm{Sym}^{t}(V) \to \wedge^{s-1}W \otimes \mathrm{Sym}^{t+1}(V)$$}}
	is exact for all $(s,t) \neq (0,0)$ and so $\mathrm{K}_{p,q}(S_V, W)=0$ for $(p,q) \neq (0,0)$.
	
	 By the assumption that $(X,L)$ is normally generated, we have an exact sequence
	$$0 \to \Gamma_{\PP}(\mathcal{O}_{\PP}(1), I_{X/\PP}) \to \Gamma_{\PP}(\mathcal{O}_{\PP}(1)) \to \Gamma_{X}(L) \to 0$$
	of $S_V=\Gamma_{\PP}(\mathcal{O}_{\PP}(1))$ modules. We obtain an exact sequence
	{\small{$$\to \mathrm{K}_{p,1}(S_V,V) \to \mathrm{K}_{p,1}(X,L) \to \mathrm{K}_{p-1,2}(\PP,I_{X/\PP},\mathcal{O}_{\PP}(1)) \to \mathrm{K}_{p-1,2}(S_V,V) \to.$$}}
	Thus $\mathrm{K}_{p,1}(X,L) \simeq \mathrm{K}_{p-1,2}(\PP,I_{X/\PP},\mathcal{O}_{\PP}(1))$. Likewise, considering the short exact sequence as a sequence of $S_W$ modules, $\mathrm{K}_{p,1}(X,L; W) \simeq \mathrm{K}_{p-1,2}(\PP,I_{X/\PP},\mathcal{O}_{\PP}(1); W)$.
	
	It remains to show that the natural map $\mathrm{K}_{p-1,2}(\PP,I_{X/\PP},\mathcal{O}_{\PP}(1); W) \to \mathrm{K}_{p-1,2}(\PP,I_{X/\PP},\mathcal{O}_{\PP}(1))$ is injective. Since $\mathrm{H}^0( \mathcal{O}_{\PP}(1) \otimes I_{X/\PP})=0$ 
	{\small{\begin{align*}
	\mathrm{K}_{p-1,2}(\PP,I_{X/\PP},\mathcal{O}_{\PP}(1); W)&=\mathrm{Ker}(\wedge^{p-1}W \otimes \mathrm{H}^0( \mathcal{O}_{\PP}(2) \otimes I_{X/\PP})
	\to \wedge^{p-2}W \otimes \mathrm{H}^0( \mathcal{O}_{\PP}(3) \otimes I_{X/\PP}),)\\
	\mathrm{K}_{p-1,2}(\PP,I_{X/\PP},\mathcal{O}_{\PP}(1))&=\mathrm{Ker}(\wedge^{p-1}V \otimes \mathrm{H}^0( \mathcal{O}_{\PP}(2) \otimes I_{X/\PP})
	\to \wedge^{p-2}V \otimes \mathrm{H}^0( \mathcal{O}_{\PP}(3) \otimes I_{X/\PP}))
	\end{align*}}}
	and the claim follows from the inclusion $\wedge^{p-1}W \seq \wedge^{p-1}V$.
	\end{proof}

The following proposition will be useful.
\begin{prop} \label{elementary-van}
	Let $(X,L)$ be a projective variety and let $W \seq \mathrm{H}^0(X,L)$ be a space of sections generating $L$. Let $F$ be a coherent sheaf of $X$. Then $\mathrm{K}_{p,q}(X,F,L; W)=0$ for $p \geq \dim W-1$ and any $q$.
	\end{prop}
\begin{proof}
	Let $M_W$ be the kernel bundle fitting into the exact sequence
	$$0 \to M_W \to W \otimes \mathcal{O}_X \xrightarrow{ev} L \to 0,$$
	where $W \otimes \mathcal{O}_X \xrightarrow{ev} L$ is the evaluation map. We have 
	{\small{$$\mathrm{K}_{p,q}(X,F,L;W)=\mathrm{Ker} \left( \mathrm{H}^1(X,\bigwedge^{p+1}M_W \otimes F \otimes L^{q-1}) \to \bigwedge^{p+1} W \otimes \mathrm{H}^1(X,F \otimes L^{q-1})  \right), $$}}
	\cite[Remark 2.6]{aprodu-nagel}. Since $\mathrm{rk}(M_W)=\dim W -1$, the claim follows.
	\end{proof}

We may now state the definition of rank of a syzygy.
\begin{defin} \label{rank-defin} Let $(X,L)$ be normally generated and $\alpha \neq 0 \in \mathrm{K}_{p,1}(X,L)$. We define the \textbf{rank} of $\alpha$ as the dimension of the smallest subspace $V \seq \mathrm{H}^0(X,L)$ such that $\alpha \in \mathrm{K}_{p,1}(X,L; V)\seq \mathrm{K}_{p,1}(X,L)$.
	\end{defin}
Our definition of rank comes from \cite[Chapter 3]{aprodu-nagel}. When $X$ is integral, for any $\alpha \neq 0 \in \mathrm{K}_{p,1}(X,L)$ we have the inequality
$\rank(\alpha) \geq p+1,$ \cite{bothmer-JPAA}. The rank of a syzygy may alternatively be described in terms of the maps in the linear free resolution of $\Gamma_X(L)$. If 
 $$\ldots \to F_2 \xrightarrow{\partial_2} F_1 \xrightarrow{\partial_1} F_0 \to \Gamma_X(L) \to 0$$
 is the minimal free resolution, then $F_i \simeq \oplus_{j \geq 1} K_{i,j}(X,L) \otimes_{\C} S(-i-j)$
 for $i \geq 1$. Restricting the $p^{th}$ map $$\partial_p \; : \; F_p \to F_{p-1}$$ in the resolution to $K_{p,1}(X,L)\otimes S(-p-1)$ and composing with the projection
 $F_{p-1} \to K_{p-1,1}(X,L)\otimes S(-p)$, we obtain a map $\mathrm{K}_{p,1}(X,L)\otimes S(-p-1) \to \mathrm{K}_{p-1,1}(X,L)\otimes S(-p).$
 Taking the strand of degree $p+1$ we obtain a linear map
 $$\partial^{\ell}_p \; : \; \mathrm{K}_{p,1}(X,L) \to \mathrm{K}_{p-1,1}(X,L) \otimes_{\C} \mathrm{H}^0(X,L),$$
 of vector spaces (this map also appears in \cite[Lemma 2.2]{aprodu-higher}). 
 
 
 We may describe the rank of a syzygy in terms of the complex valued matrix $\partial^{\ell}_p$. Indeed this was von Bothmer's original definition of rank of a syzygy \cite[Definition 3.9]{bothmer-Transactions}. Recall that we have a natural identification
 {\small{$$\mathrm{K}_{p,1}(X,L)\simeq \mathrm{Ker}(\bigwedge^{p-1}\mathrm{H}^0(X,L) \otimes I_2 \to \bigwedge^{p-2} \mathrm{H}^0(X,L) \otimes I_3),$$}}
 from the proof of Proposition \ref{inc-subspace}, where $I:=\mathrm{Ker}(S \to \Gamma_X(L))$ is the homogeneous ideal of $X$.
 \begin{thm}[\cite{eisenbud-goto}, Appendix to Section 1] \label{EG}
Let $\alpha=\sum_i \sigma_i \otimes Q_i \in \mathrm{Ker}(\bigwedge^{p-1}\mathrm{H}^0(X,L) \otimes I_2 \to \bigwedge^{p-2} \mathrm{H}^0(X,L) \otimes I_3)$. Then 
{\small{$$\partial^{\ell}_p(\alpha)=\sum_i d\sigma_i \otimes Q_i \in \wedge^{p-1}\mathrm{H}^0(X,L) \otimes \mathrm{H}^0(X,L) \otimes I_2,$$}}
where $d: \wedge^{p-1}\mathrm{H}^0(X,L) \to \wedge^{p-1}\mathrm{H}^0(X,L) \otimes \mathrm{H}^0(X,L) $ is the differential.
 \end{thm}
 We first need a simple lemma.
 \begin{lem} \label{basic-case-rank-map}
 Let $V$ be a finite-dimensional, complex vector space with subspace $W \seq V$. Let $d: \wedge^p V \to \wedge^{p-1}V\otimes V$ be the differential and let $\sigma \in \wedge^pV$.
 Then $d\sigma \in \wedge^{p-1} V \otimes W$ if and only if $\sigma \in \wedge^{p}W$.
 \end{lem}
 \begin{proof}
 Let $\{e_1, \ldots e_m \}$, where $m=\dim W$, be a basis for $W$ and extend it to a basis $\{e_1, \ldots, e_n\}$, $n=\dim V$ of $V$. For all integers $\ell$ let $I_{\ell}\seq \{1, \ldots, m\}^{\times \ell}$ denote the set of tuples $\bar{v}=(v_1,\ldots, v_{\ell})$ with $v_1 < v_2 \ldots <v_{\ell}$. For $\bar{v}\in \{1, \ldots, m\}^{\times \ell}$, we use the notation $e_{\bar{v}}$ for the element $e_{v_1} \wedge \ldots \wedge e_{v_{\ell}}$, where $v_i$ denotes the $i^{th}$ component of $\bar{v}$. We denote by $\{\bar{v} \} \seq \{1, \ldots, m\}$ the set $\{v_1, \ldots, v_{\ell} \}$ and define {\small{$$\bar{v}_j :=(v_1, \ldots, v_{j-1}, v_{j+1}, \ldots v_{\ell}) \in I_{\ell-1}.$$}}
 The set $\{e_{\bar{v}} \; | \; \bar{v} \in I_p \}$ forms a basis for $\wedge^p V$. Let $\sigma=\sum_{\bar{v} \in I_p}\alpha_{\bar{v}} \, e_{\bar{v}} \in \wedge^p V,$ with $\alpha_{\bar{v}} \in \C$ for all $\bar{v} \in I_p$. Then
{\small{ \begin{align*}
 d\sigma=\sum_{\bar{v} \in I_p} \sum_{j=1}^p (-1)^j \alpha_{\bar{v}} \, e_{\bar{v}_j}\otimes e_{v_j} =\sum_{\bar{u} \in I_{p-1}} e_{\bar{u}} \otimes \sum_{i \in \{1,\ldots, m\}\setminus \{ \bar{u} \}} \pm \alpha_{\bar{u}\cup\{i\}}e_i,
  \end{align*}}}
  where, for $i \in \{1,\ldots, m\}\setminus \{ \bar{u} \}$, we denote by $\bar{u}\cup\{i\} \in I_{p}$ the unique element in $I_p$ with underlying set $\{\bar{u}\cup\{i\} \}=\{\bar{u}\} \cup \{i\}$. Thus $d\sigma \in \wedge^{p-1} V \otimes W$ if and only if $\alpha_{\bar{v}}=0$ for all $\bar{v}\in I_p$ with a component $v_i >m$ which occurs if and only if $\sigma \in \wedge^pW$.
 \end{proof}
 We may now characterize the rank of a syzygy in terms of the map  $\partial^{\ell}_p : \mathrm{K}_{p,1}(X,L) \to \mathrm{K}_{p-1,1}(X,L) \otimes_{\C} \mathrm{H}^0(X,L)$.
 \begin{prop}
 Let $(X,L)$ be a normally generated, polarized variety. Let $\alpha \neq 0 \in \mathrm{K}_{p,1}(X,L)$. Then $\alpha$ has rank $\leq r$ if and only if there exists a vector space $V \seq  \mathrm{H}^0(X,L)$ of dimension at most $r$ and $\partial^{\ell}_p(\alpha) \in  \mathrm{K}_{p-1,1}(X,L) \otimes_{\C} V$.
 \end{prop}
 \begin{proof}
 By definition $\alpha$ has rank $\leq r$ if and only if there exists a subspace $V \seq \mathrm{H}^0(X,L)$ of dimension at most $r$ with $\alpha \in \mathrm{K}_{p,1}(X,L;V)$. As in the proof of Proposition \ref{inc-subspace}, we have a natural identification $\mathrm{K}_{p,1}(X,L)\simeq \mathrm{Ker}(\wedge^{p-1}\mathrm{H}^0(X,L) \otimes I_2 \to \wedge^{p-2} \mathrm{H}^0(X,L) \otimes I_3),$
and $\alpha$ has rank $\leq r$ if and only if there exists a subspace $V \seq \mathrm{H}^0(X,L)$ of dimension at most $r$ with $\alpha \in \wedge^{p-1}V \otimes I_2$. By Lemma \ref{basic-case-rank-map} and Theorem \ref{EG}, this is equivalent to $\partial^{\ell}_p(\alpha) \in \wedge^{p-1}\mathrm{H}^0(X,L) \otimes V \otimes I_2$.
 \end{proof}

The rank of a syzygy $\alpha$ can be thought of as a measure of complexity. Syzygies of minimal rank $p+1$ are the simplest, and all such syzygies can be constructed explicitly \cite{bothmer-JPAA}, \cite[CH.\ 3]{aprodu-nagel}. Precisely, all such syzygies are the restriction of syzygies of a rational normal scroll. 
 Alternatively, at least if $X$ is smooth, all minimal rank syzygies arise from the Green--Lazarsfeld construction of syzygies, \cite[Appendix A]{green-koszul}. Indeed, if $\alpha \neq 0 \in \mathrm{K}_{p,1}(X,L; W)$ with $\dim W =p+1$, then by Proposition \ref{elementary-van}, $W$ cannot be base-point free. The base-locus contains a divisor $D \seq X$ moving in a pencil and  $\alpha$ arises from the Green--Lazarsfeld construction applied to $D, L(-D)$.
\begin{prop} \label{one-dim-iso}
	Let $(X,L)$ be a projective variety. Let $M$ be a line bundle on $X$ with $h^0(X,M)=2$ and $h^0(X,L\otimes M^{\vee})=p+1$. Pick $s \neq 0  \in \mathrm{H}^0(M)$. Assume $L \otimes M^{\vee}$ is base-point free. Then we have a natural isomorphism
	{\small{$$\mathrm{K}_{p,1}(X,L; \mathrm{H}^0(L \otimes M^{\vee}))\simeq \bigwedge^{p+1}\mathrm{H}^0(L \otimes M^{\vee}))\otimes \frac{\mathrm{H}^0(M)} {\C \langle s\rangle }.$$}}
	In particular, $\dim \mathrm{K}_{p,1}(X,L; \mathrm{H}^0(L \otimes M^{\vee}))=1.$ 
	\end{prop}
\begin{proof}
	The syzygy space $\mathrm{K}_{p,1}(X,L; \mathrm{H}^0(L \otimes M^{\vee}))$ is the middle cohomology of 
	{\small{$$\bigwedge^{p+1}\mathrm{H}^0(L \otimes M^{\vee}) \xrightarrow{f} \bigwedge^{p}\mathrm{H}^0(L \otimes M^{\vee})\otimes \mathrm{H}^0(L) \xrightarrow{g} \bigwedge^{p-1}\mathrm{H}^0(L \otimes M^{\vee})\otimes \mathrm{H}^0(L^2).$$}}
    The map 
    $\wedge^{p+1}\mathrm{H}^0(L \otimes M^{\vee}) \xrightarrow{f} \wedge^{p}\mathrm{H}^0(L \otimes M^{\vee})\otimes \mathrm{H}^0(L) \seq \wedge^{p}\mathrm{H}^0(L)\otimes \mathrm{H}^0(L)$
    factors through $d: \wedge^{p+1}\mathrm{H}^0(L ) \to \wedge^{p}\mathrm{H}^0(L)\otimes \mathrm{H}^0(L)$ via the inclusion $\mathrm{H}^0(L \otimes M^{\vee}) \seq \mathrm{H}^0(X,L)$ induced by $s$. The differential $d$ is injective since if $\wedge : \wedge^{p}\mathrm{H}^0(L)\otimes \mathrm{H}^0(L) \to \wedge^{p+1}\mathrm{H}^0(L )$ is the wedge map then $\wedge \circ d=(p+1) Id$. Thus $f$ is injective and we have an isomorphism
    {\small{$$\mathrm{K}_{p,1}(X,L; \mathrm{H}^0(L \otimes M^{\vee}))\simeq \frac{ \mathrm{Ker}(g)}{ \wedge^{p+1}\mathrm{H}^0(L \otimes M^{\vee})}.$$}}
    
    We have $\mathrm{K}_{p,1}(X,M, L \otimes M^{\vee})=0$ by Proposition \ref{elementary-van}. Thus the sequence
   {\small{ $$\bigwedge^{p+1}\mathrm{H}^0(L \otimes M^{\vee})\otimes \mathrm{H}^0(M) \xrightarrow{f'} \bigwedge^{p}\mathrm{H}^0(L \otimes M^{\vee})\otimes \mathrm{H}^0(L) \xrightarrow{g'} \bigwedge^{p-1}\mathrm{H}^0(L \otimes M^{\vee})\otimes \mathrm{H}^0(L^2\otimes M^{\vee})$$ }}
    is exact. Further $\mathrm{K}_{p+1,0}(X,M, L \otimes M^{\vee})=\mathrm{Ker}(f')$ since $\wedge^{p+2}\mathrm{H}^0(L \otimes M^{\vee})=0$. Thus $f'$ is injective by Proposition \ref{elementary-van}. Thus $\mathrm{Ker}(g)=\mathrm{Ker}(g')\simeq \wedge^{p+1}\mathrm{H}^0(L \otimes M^{\vee})\otimes \mathrm{H}^0(M).$
    This gives 
        {\small{$$\mathrm{K}_{p,1}(X,L; \mathrm{H}^0(L \otimes M^{\vee}))\simeq \frac{  \wedge^{p+1}\mathrm{H}^0(L \otimes M^{\vee})\otimes \mathrm{H}^0(M)}{ \wedge^{p+1}\mathrm{H}^0(L \otimes M^{\vee})},$$}}
        as required.
\end{proof}
	
We now recall some facts about rational normal scrolls, \cite{	schreyer1}. Consider the locally free sheaf
$$\mathcal{E}=\mathcal{O}_{\PP^1}(e_1) \oplus \ldots \oplus \mathcal{O}_{\PP^1}(e_d)$$
of rank $d$ on $\PP^1$ with the assumptions
$$e_1 \geq e_2 \geq \ldots \geq e_d \geq 0\;  \; \text{and} \;\; f:=e_1 +\ldots +e_d \geq 2.$$
Set $X:=\PP(\mathcal{E})$ and let $\pi: X \to \PP^1$ denote the projection. Then $\mathcal{O}_X(1)$ is base point free and the associated morphism $j: X \to \PP^r:=\PP$, $r=f+d-1$,
has image $j(X) \seq \PP^r$ a rational normal scroll of minimal degree $f$ in $\PP^r$, \cite{eisenbud-harris-minimal}. We let $H:=\mathcal{O}_X(1)$ and $R:=\pi^*\mathcal{O}_{\PP^1}(1)$. Then $H$ and $R$ generate the integral Picard group of $X$.
\smallskip

 The syzygy spaces $\mathrm{K}_{p,q}(X,H)$ may be described explicitly, \cite{schreyer1}. In particular, 
$b_{p,q}(X,H)=0$ for $q \geq 2$, whereas $b_{p,1}(X,H)=p\binom{f}{p+1}$. It follows readily from the fact that $H$ and $R$ generate $\mathrm{Pic}(X)$ that $H-R$ is base-point free for $f \geq 2$. We have $h^0(X,H-R)=h^0(X,H)-d=f$ and $h^0(X,R)=2$. Proposition \ref{one-dim-iso} shows that
{\small{$$\wedge^{f}\mathrm{H}^0(X,H-R) \otimes \frac{\mathrm{H}^0(X,R)}{<s>} \simeq \mathrm{K}_{f-1,1}(X,H; \mathrm{H}^0(X,H-R)) .$$}}
Note further that it makes sense to talk about the rank of syzygies in $ \mathrm{K}_{p,1}(X,H)$, even if $H$ is not very ample, since we have a natural isomorphism $ \mathrm{K}_{p,1}(X,H)\simeq \mathrm{K}_{p,1}(j(X),\mathcal{O}_{j(X)}(1))$ and the scroll $j(X) \seq \PP^r$ is normally generated.
\begin{lem} \label{rank-scroll}
Let $X:=\PP(\mathcal{E})$, for $\mathcal{E}=\mathcal{O}_{\PP^1}(e_1) \oplus \ldots \oplus \mathcal{O}_{\PP^1}(e_d)$ with 
$$e_1 \geq e_2 \geq \ldots \geq e_d \geq 0\;  \; \text{and} \;\; f:=e_1 +\ldots +e_d \geq 2.$$
We have a well-defined morphism $\phi : \PP\left(\mathrm{H}^0(\mathcal{O}_X(R))\right) \to \PP( \mathrm{K}_{f-1,1}(X,H))$ with
{\small{\begin{align*}
\phi ([s]):= [\,\wedge^{f}\mathrm{H}^0(X,H-R) \otimes \frac{\mathrm{H}^0(X,R)}{<s>}\,].
\end{align*}}}
Further, $\phi^*\mathcal{O}_{\PP( \mathrm{K}_{f-1,1}(X,H))}(1)\simeq \mathcal{O}_{\PP\left(\mathrm{H}^0(\mathcal{O}_X(R))\right) }(f-2).$ 
\end{lem}
\begin{proof}
Set $\PP:=\PP\left(\mathrm{H}^0(\mathcal{O}_X(R))\right)$.  The dual of the evaluation morphism $\mathrm{H}^0(\mathcal{O}_{\PP}(1)) \otimes \mathcal{O}_{\PP} \twoheadrightarrow \mathcal{O}_{\PP}(1)$ gives an inclusion $i: \mathcal{O}_{\PP}(-1) \hookrightarrow \mathrm{H}^0(\mathcal{O}_X(R)) \otimes \mathcal{O}_{\PP}$. At the fibre $i \otimes k([s])$ at $[s] \in \PP$ is the inclusion $<s> \seq \mathrm{H}^0(X,R)$. Let $F:=\mathrm{Coker}(i)\simeq \mathcal{O}_{\PP}(1)$ and consider the line bunde
{\small{$$L:=\wedge^f \mathrm{H}^0(X,H-R) \otimes F(1-f),$$}}
which has degree $2-f$. We have the differential
{\small{$$\wedge^f \mathrm{H}^0(X,H-R)\otimes \mathrm{H}^0(X,R) \to \wedge^{f-1} \mathrm{H}^0(X,H-R)\otimes \mathrm{H}^0(X,H).$$}}
Twisting by $\mathcal{O}_{\PP}(1-f)$ we obtain a morphism
{\small{$$\wedge^f \mathrm{H}^0(X,H-R)\otimes \mathrm{H}^0(X,R)\otimes \mathcal{O}_{\PP}(1-f) \to \wedge^{f-1}( \mathrm{H}^0(X,H-R)\otimes \mathcal{O}_{\PP}(-1) )\otimes \mathrm{H}^0(X,H).$$}}
Using the composition 
{\small{$ \mathrm{H}^0(H-R)\otimes \mathcal{O}_{\PP}(-1) \to  \mathrm{H}^0(H-R)\otimes\mathrm{H}^0(R) \otimes \mathcal{O}_{\PP}\to \mathrm{H}^0(H) \otimes \mathcal{O}_{\PP},$}}
we get a morphism
{\small{$$\wedge^f \mathrm{H}^0(X,H-R)\otimes \mathrm{H}^0(X,R)\otimes \mathcal{O}_{\PP}(1-f) \to \wedge^{f-1}\mathrm{H}^0(X,H)\otimes \mathrm{H}^0(X,H)\otimes \mathcal{O}_{\PP},$$}}
which induces a morphism
{\small{$$\delta \; :\;  \wedge^f \mathrm{H}^0(X,H-R)\otimes \mathrm{H}^0(X,R)\otimes \mathcal{O}_{\PP}(1-f) \to \frac{\wedge^{f-1}\mathrm{H}^0(X,H)\otimes \mathrm{H}^0(X,H)}{\wedge^f \mathrm{H}^0(X,H)}\otimes \mathcal{O}_{\PP}.$$}}
By the proof of Proposition \ref{one-dim-iso}, $\delta$ vanishes on the subbundle $\wedge^f \mathrm{H}^0(X,H-R)\otimes \mathcal{O}_{\PP}(-f)$ and hence induces a morphism 
{\small{$$\delta' \; :\;  L \to \frac{\wedge^{f-1}\mathrm{H}^0(X,H)\otimes \mathrm{H}^0(X,H)}{\wedge^f \mathrm{H}^0(X,H)}\otimes \mathcal{O}_{\PP},$$}}
which lands in $\mathrm{K}_{f-1,1}(X,H) \otimes \mathcal{O}_{\PP}$. The resulting morphism
{\small{$$\widetilde{\delta}\; : \; L \to \mathrm{K}_{f-1,1}(X,H)\otimes  \mathcal{O}_{\PP}$$}}
has fibres $\widetilde{\delta} \otimes k([s])$ which are naturally identified with the natural map
{\small{$$\wedge^{f}\mathrm{H}^0(X,H-R) \otimes \frac{\mathrm{H}^0(X,R)}{<s>} \to \mathrm{K}_{f-1,1}(X,H),$$}}
which is injective by Proposition \ref{one-dim-iso} and Proposition \ref{inc-subspace}, noting that we may naturally identify $ \mathrm{K}_{p,1}(X,H)$ with $\mathrm{K}_{p,1}(j(X),\mathcal{O}_{j(X)}(1))$, since $j(X)$ has rational singularities and is normally generated  \cite{schreyer1}. The dual $\delta^{\vee}$ is a surjective morphism of vector bundles and induces the desired morphism $\phi : \PP\left(\mathrm{H}^0(\mathcal{O}_X(R))\right) \to \PP( \mathrm{K}_{f-1,1}(X,H))$, \cite[II.7.12]{hartshorne}. 
\end{proof}
We end this section with a description of the last syzygy space of a scroll. 
\begin{prop}[cf.\ \cite{bothmer-Transactions}, Prop.\ 4.3.] \label{rank-scroll-2}
The morphism $\phi$ from Lemma \ref{rank-scroll} induces an isomorphism $$\phi^* \; : \; \mathrm{H}^0(\mathcal{O}_{ \PP(\mathrm{K}_{f-1,1}(X,H))}(1)) \xrightarrow{\sim} \mathrm{H}^0(\mathcal{O}_{\PP(\mathrm{H}^0(\mathcal{O}_X(R)))}(f-2)).$$
\end{prop}
\begin{proof}
By the Eagon--Northcott complex resolving $j(X) \seq \PP^r$, we have an isomorphism $$ \mathrm{K}_{f-1,1}(X,H) \simeq \mathrm{Sym}^{f-2}\left(\mathrm{H}^0(\mathcal{O}_X(R))\right),$$ \cite{schreyer1}. Thus $h^0(\mathcal{O}_{ \PP(\mathrm{K}_{f-1,1}(X,H))}(1)) = h^0(\mathcal{O}_{\PP(\mathrm{H}^0(\mathcal{O}_X(R)))}(f-2))$ and it suffices to show that $\phi^*$ is injective, i.e.\ that the image of $\phi$ is not contained in any hyperplane. By Proposition \ref{one-dim-iso}  and the definition of $\phi$, it suffices to show that $\mathrm{K}_{f-1,1}(X,H)$ is spanned by the one-dimensional subspaces $\mathrm{K}_{f-1,1}(X,H; \mathrm{H}^0(X,H\otimes I_{Z(\beta)}))$, $\beta \in \mathrm{H}^0(\mathcal{O}_X(R))$. But von Bothmer shows that the element $\beta^{f-2} \in \mathrm{Sym}^{f-2}\left(\mathrm{H}^0(\mathcal{O}_X(R))\right) \simeq  \mathrm{K}_{f-1,1}(X,H)$ lies in $\mathrm{K}_{f-1,1}(X,H; \mathrm{H}^0(X,H\otimes I_{Z(\beta)}))$ for $\beta \in \mathrm{H}^0(\mathcal{O}_X(R))$. Since such elements span $\mathrm{K}_{f-1,1}(X,H)$, this completes the proof.
\end{proof}
By Proposition \ref{rank-scroll-2}, $\mathrm{K}_{f-1,1}(X,H)$ is spanned by the rank $f$ syzygies
$$\alpha \in \mathrm{K}_{f-1,1}(X,H; \mathrm{H}^0(X,H\otimes I_{Z(s)})), \; \; \text{for}\; \; s \in \mathrm{H}^0(\mathcal{O}_X(R)).$$

\section{Brill--Noether Theory for Integral Curves Lying on a Surface} \label{BN-sing-curves}
Let $C$ be an integral, projective, curve which admits an embedding into a smooth surface, or equivalently, each singularity of $C$ has embedding dimension at most two, \cite[Corollary 9]{altman-kleiman-bertini}. Consider the compactified Jacobian $\bar{J}^d(C)$ of rank one, torsion free sheaves of degree $d$ on $C$, \cite{rego}, \cite{altman-kleiman}. This compactification is often attributed to Mumford \cite{mumford} and Mayer \cite{mayer} and also works in a relative setting. Then $\bar{J}^d(C)$ is a projective, integral, local complete intersection scheme of dimension $g$. The Picard scheme $\mathrm{Pic}^d(C)$ of degree $d$ line bundles is a dense, open subset of $\bar{J}^d(C)$. Let $\mathrm{Hilb}^d(C)$ denote the Hilbert scheme parametrizing zero-dimensional closed subschemes $Z \seq C$ of degree $d$. From \cite{altman-kleiman}, we have the \emph{Abel map}
$$\rho \; : \; \mathrm{Hilb}^d(C) \to \bar{J}^d(C)$$
taking the closed subscheme $Z \seq C$ to the torsion-free sheaf $I^{\vee}_Z:=\mathcal{H}om_{\mathcal{O}_C}(I_Z,\mathcal{O}_C)$, where $I_Z$ is the ideal sheaf of $Z$.\smallskip

Note that $C$ is Gorenstein. All torsion-free sheaves on an integral, Gorenstein curve are reflexive, \cite[Lemma 1.1]{hartshorne-generalized}. In particular, for each $[Z] \in \mathrm{Hilb}^d(C)$, the ideal sheaf $I_Z$ is reflexive, as is $A:=\rho([Z])=I_Z^{\vee}.$ We may thus identify $\rho^{-1}([A])$ with the projective space $\PP(\mathrm{Hom}_C(A^{\vee}, \mathcal{O}_C))$, since any nonzero homomorphism between rank one, torsion-free sheaves on $C$ is injective. Dualizing each morphism in $\mathrm{Hom}_C(A^{\vee}, \mathcal{O}_C)$ we identify $\rho^{-1}([A])$ with $\PP(\mathrm{H}^0(C,A))$.\smallskip

There is a surjective, \'etale morphism $\nu:  \widetilde{J}^d(C) \to  \bar{J}^d(C)$ and a universal, torsion-free sheaf $\mathcal{A}$ of rank one on $C \times \widetilde{J}^d(C) $, i.e.\ $\mathcal{A}_{|_{C \times \{x\}}}\simeq \nu(x) \in \bar{J}^d(C).$ As in \cite{bhosle-param}, we have the Brill--Noether loci
$$W^r_d(C):= \{ [A] \in \bar{J}^d(C) \; | \; h^0(A) \geq r+1 \},$$
which are the image of determinantal subschemes of $\widetilde{J}^d(C)$ under $\nu$. More precisely, as in \cite[\S 2.2]{bhosle-param}, there is a determinantal subscheme $\widetilde{W}^r_d(C) \seq \widetilde{J}^d(C)$ with 
$$\widetilde{W}^r_d(C):= \{ x \in \widetilde{J}^d(C) \; | \; h^0(\nu(x)) \geq r+1 \}.$$
We set $W^r_d(C):= \nu(\widetilde{W}^r_d(C))$. Let $\tilde{p}: C \times  \widetilde{J}^d(C)  \to C$  denote the first projection and denote by $\tilde{q}: C \times  \widetilde{J}^d(C)  \to \widetilde{J}^d(C) $ the second projection and let $q: C \times  \widetilde{J}^d(C)  \to  \bar{J}^d(C)$ denote the composition $q:=\nu \circ \tilde{q}$.
We can construct
$$\widetilde{G}^r_d(C):=\{ (x,V) \; | x \in \widetilde{W}^r_d(C), \; V \in \mathrm{Gr}(r+1,\mathrm{H}^0(\nu(x))) \}$$
 \cite[IV.3]{ACGH1}, and we set 
$$G^r_d(C):=\nu(\widetilde{G}^r_d(C)).$$

We next generalize the base-point free pencil trick, \cite[\S III.3]{ACGH1} to torsion-free sheaves.
	\begin{lem} \label{tf-bpf}
	Let $C$ be an integral curve lying on a smooth surface and let $[A] \in W^1_d(C)\setminus W^2_d(C)$ be globally generated. Then we have an exact sequence
	$$0 \to A^{\vee} \xrightarrow{i} \mathrm{H}^0(A) \otimes O_C \xrightarrow{ev} A \to 0,$$
	where $\mathrm{ev} : \mathrm{H}^0(A) \otimes O_C \twoheadrightarrow A$ is the evaluation morphism, and where $i$ is given by the formula
	$$i(s)=y\otimes \rho(xs)-x\otimes \rho(ys),$$
	for $x,y$ a basis of $\mathrm{H}^0(A)$ and $\rho: \mathrm{H}^0(A) \otimes A^{\vee} \to \mathcal{O}_C$ the natural map.
	\end{lem}
	\begin{proof}
	Let $K:=\mathrm{Ker}(\mathrm{ev})$. Then $K$ is a rank-one, torsion free sheaf. We need to show $K \simeq A^{\vee}$. Let $Z:=Z(x)$  be the zero scheme of $x \in \mathrm{H}^0(A)$. The claim now follows by observing the commutative diagram
	 {\small{$$\begin{tikzcd}
	 	& & 0 \arrow[d] & 0 \arrow[d] &\\
& & \mathcal{O}_C \arrow[r, "\sim"] \arrow[d, "\alpha"] & \mathcal{O}_C \arrow[d] &\\
	0 \arrow[r] & K \arrow[r] \arrow{d}{\simeq}[swap] {\psi} & \mathrm{H}^0(A) \otimes \mathcal{O}_C \arrow[r] \arrow[d, "\phi"] &A \arrow[r] \arrow[d] & 0\\
	0 \arrow[r] & I_Z \simeq A^{\vee} \arrow[r, "j"] &  \mathcal{O}_C \arrow[r] \arrow[d] &\mathcal{O}_Z \arrow[r] \arrow[d] & 0\\
		 	& & 0  & 0 &
\end{tikzcd}$$}}
	with exact rows and columns, where $\alpha$ is the inclusion of the coordinate $x\in \mathrm{H}^0(A)$,  $\phi$ is the projection to the coordinate $y \in \mathrm{H}^0(A)$, $j: A^{\vee} \to \mathcal{O}_C$ is multiplication by $x$ and $\psi=\phi_{|_K}$. Clearly the sheaf morphism $i: A^{\vee} \to  \mathrm{H}^0(A) \otimes \mathcal{O}_C$ lands in the subsheaf $K$, and we have $\phi \circ i=j$, i.e.\ $\psi \circ i=id$. Thus $i$ is the inverse to the isomorphism $\psi$, as required.
	\end{proof}
	We now need a version of Proposition \ref{one-dim-iso} with $M$ assumed merely to be torsion-free.
\begin{lem} \label{natural-injection}
Let $C$ be an integral curve which lies on a smooth surface and let $[A]\in W^1_d(C) \setminus W^2_d(C)$. Choose $s \neq 0 \in \mathrm{H}^0(A)$. For any $p$, we have a natural map
{\small{$$\delta_s \; : \; \bigwedge^{p+1}\mathrm{H}^0(\omega_C \otimes A^{\vee}) \otimes \frac{ \mathrm{H}^0(A)}{\C \langle s\rangle } \to \mathrm{K}_{p,1}(C,\omega_C; \mathrm{H}^0(\omega_C \otimes A^{\vee})).$$}}
\end{lem}
\begin{proof}
We have a morphism $A^{\vee} \otimes A \to \mathcal{O}_C$ of sheaves, as well as a natural composition $$\mathrm{H}^0(\omega_C \otimes A^{\vee}) \otimes \mathrm{H}^0(A)\to\mathrm{H}^0(\omega_C \otimes A^{\vee} \otimes A) \to \mathrm{H}^0( \omega_C).$$ Contracting by $s$ induces an map $m_s : \mathrm{H}^{0}(\omega_C\otimes A^{\vee}) \to \mathrm{H}^{0}(\omega_C),$ which is injective since $m_s$ is induced from the inclusion $A^{\vee} \hookrightarrow \mathcal{O}_C$ dual to $s: \mathcal{O}_C \to A$. We have a natural inclusion $\mathrm{H}^0(\omega_C \otimes A^{\vee}) \seq \mathrm{H}^0(\omega_C)$ and $\mathrm{K}_{p,1}(C,\omega_C; \mathrm{H}^0(\omega_C \otimes A^{\vee}))$ is well defined for $A$ torsion free.\smallskip

 We have the natural complex
{\small{$$\bigwedge^{p+1}\mathrm{H}^0(\omega_C \otimes A^{\vee}) \otimes \mathrm{H}^0(A) \xrightarrow{d_1} \bigwedge^{p}\mathrm{H}^0(\omega_C \otimes A^{\vee})  \otimes \mathrm{H}^0(\omega_C) \xrightarrow{d_2} \bigwedge^{p-1}\mathrm{H}^0(\omega_C \otimes A^{\vee})  \otimes \mathrm{H}^0(\omega_C^{\otimes ^2} \otimes A^{\vee} ).$$}}
So we have a map 
{\small{$$d_1 \; : \; \bigwedge^{p+1}\mathrm{H}^0(\omega_C \otimes A^{\vee}) \otimes \mathrm{H}^0(A) \to \mathrm{Ker} (d_2) \seq \bigwedge^{p}\mathrm{H}^0(\omega_C \otimes A^{\vee})  \otimes \mathrm{H}^0(\omega_C).$$}}
Since $\mathrm{K}_{p,1}(C,\omega_C; \mathrm{H}^0(\omega_C \otimes A^{\vee})):=\frac{\mathrm{Ker} (d_2)}{d_1(\bigwedge^{p+1}\mathrm{H}^0(\omega_C \otimes A^{\vee}) \otimes \C \langle s\rangle)}$, we have a natural map 
{\small{$$\bigwedge^{p+1}\mathrm{H}^0(\omega_C \otimes A^{\vee}) \otimes \frac{ \mathrm{H}^0(A)}{\C \langle s\rangle } \to \mathrm{K}_{p,1}(C,\omega_C; \mathrm{H}^0(\omega_C \otimes A^{\vee})).$$}}
\end{proof}
\begin{remark} If $A$ is locally free and $p=g-d$ then $h^0(\omega_C\otimes A^{\vee})=p+1$ by Riemann--Roch and Serre duality. If in addition $(C,\omega_C)$ is normally generated, then $\delta_s$ coincides with the isomorphism from Proposition \ref{one-dim-iso}.
	\end{remark}
\begin{prop} \label{grauert-replacement}
Let $C$ be an integral curve lying on a surface and fix $d,r \geq 0$. Suppose that we have the equality $h^0(A)=r+1$ for all points $y=[A] \in W^r_d(C)$. Then $\tilde{q}_* \mathcal{A}$ and $\tilde{q}_* (\tilde{p}^*\omega_C \otimes \mathcal{A}^{\vee})$ are locally free sheaves with natural isomorphisms
$$\tilde{q}_* \mathcal{A}\otimes k(y) \simeq \mathrm{H}^0(C,A), \; \; \tilde{q}_* (\tilde{p}^*\omega_C \otimes \mathcal{A}^{\vee}) \otimes k(y) \simeq \mathrm{H}^0(C,\omega_C \otimes A^{\vee}).$$
\end{prop}
\begin{remark}
Note that one cannot conclude the above proposition using Grauert's Theorem, as we are not assuming that $\tilde{W}^r_d(C)$ is reduced.
\end{remark}
\begin{proof}
Under the assumption $h^0(A)=r+1$ for all $y=[A] \in W^r_d(C)$, the natural morphism $c: \widetilde{G}^r_d(C) \to \widetilde{J}^d(C)$ is a closed immersion, see \cite[IV.3, IV.4]{ACGH1}. The image of $c$ is $\widetilde{W}^r_d(C)$, so we have a natural identification $\widetilde{G}^r_d(C) \simeq \widetilde{W}^r_d(C)$ under our assumptions. By construction of $\widetilde{G}^r_d(C) \simeq \widetilde{W}^r_d(C)$, there is a locally free subsheaf $F \seq \tilde{q}_* \mathcal{A}$ such that, for any $y \in \widetilde{W}^r_d(C)$, $\nu(y)=[A] \in W^r_d(C)$, the natural composition
$$F \otimes k(y) \to \tilde{q}_* \mathcal{A}\otimes k(y) \to \mathrm{H}^0(C,A)$$
is injective. By our assumptions, this composition must be an isomorphism, and in particular the natural composition $\tilde{q}_* \mathcal{A}\otimes k(y) \to \mathrm{H}^0(C,A)$ is surjective. By the Base Change Theorem, \cite[Ch.\ III, Thm.\ 12.11]{hartshorne}, $\tilde{q}_* \mathcal{A}$ is locally free and we have an isomorphism $\tilde{q}_* \mathcal{A}\otimes k(y) \simeq \mathrm{H}^0(C,A)$. Further, since $\tilde{q}$ has one dimensional fibres, and further $R^2\tilde{q}_* \mathcal{A}=0$ by \cite[Ch.\ III, Cor.\ 11.2]{hartshorne}, the Base Change Theorem implies that the natural map  $R^1\tilde{q}_* \mathcal{A}\otimes k(y) \to \mathrm{H}^1(C,A)$ is an isomorphism. Combining this with the fact that $\tilde{q}_* \mathcal{A}\otimes k(y) \simeq \mathrm{H}^0(C,A)$, another application of the Base Change Theorem gives us that $R^1\tilde{q}_* \mathcal{A}$ is locally free. By the Relative Duality Theorem \cite[Ch.\ 2]{birkar-topics}, we have natural isomorphisms $\tilde{q}_* (\tilde{p}^*\omega_C \otimes \mathcal{A}^{\vee}) \simeq (R^1\tilde{q}_* \mathcal{A})^{\vee}$, which gives the claim (using usual Serre duality $\mathrm{H}^1(C,A)^{\vee} \simeq \mathrm{H}^0(C,\omega_C \otimes A^{\vee})$).
\end{proof}
Assume that $r,d \geq 0$ are such that $W^{r+1}_d(C)=\emptyset$. By Proposition \ref{grauert-replacement}, we have a vector bundle $\tilde{q}_*\mathcal{A}$ on $\widetilde{W}^r_d(C)$. In this situation, we define the projective bundle $$\pi: X^r_d(C) \to \widetilde{W}^r_d(C)$$ by $X^r_d(C):=\PP(\tilde{q}_*\mathcal{A}),$
where we are use the convention $\PP(\tilde{q}_*\mathcal{A}):=\text{Proj}\left((\tilde{q}_*\mathcal{A})^{\vee}\right)$, where the Proj functor is defined in \cite[Ch.\ II.7]{hartshorne}. We let 
$$p'\; : \; C \times X^r_d(C) \to C, \; \; q' \; : \; C \times X^r_d(C) \to X^r_d(C),$$
denote the projections. Also let $\pi' : C \times  X^r_d(C) \to C \times  \widetilde{W}^r_d(C)$
denote the morphism $id \times \pi$. We denote by $\mathcal{O}(1)$ the line bundle $\mathcal{O}_{\PP(\tilde{q}_*\mathcal{A})}(1)$ on $X^r_d(C)$. With our conventions, $\pi_*\mathcal{O}(1) \simeq (\tilde{q}_*\mathcal{A})^{\vee}$.
\begin{lem}
We have a natural inclusion ${\pi'}^*\mathcal{A}^{\vee} \otimes {q'}^*\mathcal{O}(-1) \hookrightarrow \mathcal{O}_{C \times X^r_d(C)}.$
\end{lem}
\begin{proof}
Dualizing, it is enough to give a canonical, nonzero, section $s$ of ${\pi'}^*\mathcal{A}\otimes {q'}^*\mathcal{O}(1)$. By the projection formula and flat base change we have natural identifications
{\small{\begin{align*}
\mathrm{H}^0(C \times  X^r_d(C),{\pi'}^*\mathcal{A}\otimes {q'}^*\mathcal{O}(1)) &\simeq \mathrm{H}^0( X^r_d(C),{q'}_*{\pi'}^*\mathcal{A}\otimes \mathcal{O}(1))\\
&\simeq \mathrm{H}^0( X^r_d(C),\pi^*\tilde{q}_*\mathcal{A}\otimes \mathcal{O}(1))\\
&\simeq \mathrm{H}^0( \widetilde{W}^r_d(C),\tilde{q}_*\mathcal{A}\otimes (\tilde{q}_*\mathcal{A})^{\vee})\\
& \simeq \text{Hom}_{\mathcal{O}_{\widetilde{W}^r_d(C)}}(\tilde{q}_*\mathcal{A}, \tilde{q}_*\mathcal{A}).
\end{align*}}}
We then take the section $s$ corresponding to the identity $id \in \text{Hom}_{\mathcal{O}_{\widetilde{W}^r_d(C)}}(\tilde{q}_*\mathcal{A}, \tilde{q}_*\mathcal{A})$.
\end{proof}

The \emph{gonality} of an integral curve $C$ lying on a surface is defined to be the minimal $d$ such that $W^1_d(C) \neq \emptyset$.
\begin{prop} \label{definability-map-brill-syz}
	Let $C$ be an integral curve which may be embedded into a surface, and fix integers $p,d \geq 0$. Assume that $\omega_C$ is very ample and $(C,\omega_C)$ is normally generated. Assume further $W^2_d(C)=\emptyset$ and that, for any $[A]\in W^1_d(C)$ and $s \neq 0 \in \mathrm{H}^0(C,A)$, the map
	$$ \delta_s \; : \; \bigwedge^{g-d+1}\mathrm{H}^0(\omega_C \otimes A^{\vee}) \otimes \frac{ \mathrm{H}^0(A)}{\C \langle s\rangle } \to \mathrm{K}_{g-d,1}(C,\omega_C; \mathrm{H}^0(\omega_C \otimes A^{\vee}))$$ from Lemma \ref{natural-injection} is injective. We have a well-defined morphism $S  :  X^1_d(C)  \to \PP(\mathrm{K}_{g-d,1}(C,\omega_C))$

	{\small{\begin{align*}
	S([(p,s)]):= [\bigwedge^{g-d+1}\mathrm{H}^0(\omega_C \otimes A^{\vee}) \otimes \frac{ \mathrm{H}^0(A)}{\C \langle s\rangle }],
	\end{align*}}}
	where $p \in \widetilde{W}^1_d(C)$ and $\nu(p)=[A] \in W^1_d(C)$.
	\end{prop}
\begin{proof}
The proof is similar to that of Lemma \ref{rank-scroll}. Define $W_s:=\wedge^{g-d+1}\mathrm{H}^0(\omega_C \otimes A^{\vee}) \otimes \frac{ \mathrm{H}^0(A)}{\C \langle s\rangle }.$ By assumption, $h^0(A)=2$ and, by Riemann--Roch and Serre duality, $h^0(\omega_C \otimes A^{\vee})=g-d+1$, so $\dim W_s=1$. We have an inclusion $\mathrm{K}_{g-d,1}(C,\omega_C; \mathrm{H}^0(\omega_C \otimes A^{\vee})) \hookrightarrow \mathrm{K}_{g-d,1}(C,\omega_C)$
by Proposition \ref{inc-subspace}. Under the assumption that $\delta_s$ is injective, we may consider $W_s$ as a subspace of $\mathrm{K}_{g-d,1}(C,\omega_C)$. The map $S$ is thus well-defined on closed points of $X^1_d(C)$.\smallskip

	To show that $S$ is defined as a morphism of schemes, we show that there is a line subbundle $i :\mathcal{L} \hookrightarrow \mathrm{K}_{g-d,1}(C,\omega_C) \otimes \mathcal{O}_{X^1_d(C)}$ such that, for any closed point $x=[(p,s)]  \in X^1_d(C)$, with $\nu(p)=[A]$, 
	{\small{$$\mathcal{L}\otimes k(x) \simeq W_s$$}}
	and further $i \otimes k(x)$ is identified with the inclusion $W_s \hookrightarrow \mathrm{K}_{p,1}(C,\omega_C)$
	from Lemma \ref{natural-injection} and Proposition \ref{inc-subspace}. The claim then follows from \cite[II.7.12]{hartshorne}. We have a natural evaluation morphism $\pi^*\pi_*\mathcal{O}(1) \twoheadrightarrow \mathcal{O}(1)$ on $X^1_d(C)=\PP(\tilde{q}_*\mathcal{A})$. Since $\pi_*\mathcal{O}(1)\simeq (\tilde{q}_*\mathcal{A})^{\vee}$, we may dualize the evaluation morphism to get an inclusion
	{\small{$$j \; : \; \mathcal{O}(-1) \hookrightarrow \pi^*\tilde{q}_*\mathcal{A},$$}}
	such that $j \otimes k(x)$ corresponds to $\C\langle s \rangle \hookrightarrow \mathrm{H}^0(A)$. Let $\mathcal{F}:=\mathrm{Coker}(j)$. We set
	{\small{$$\mathcal{L}:=\bigwedge^{g-d+1}{\pi}^*\tilde{q}_*(\tilde{p}^*\omega_C \otimes \mathcal{A}^{\vee}) \otimes \mathcal{F}(d-g).$$}}
	Then $\mathcal{L}$ is a line bundle with $\mathcal{L}\otimes k(x) \simeq W_s$.\smallskip
	
	We have the differential 
	{\small{$$\wedge^{g-d+1}\pi^*\tilde{q}_*(\tilde{p}^*\omega_C \otimes \mathcal{A}^{\vee}) \to \wedge^{g-d}\pi^*\tilde{q}_*(\tilde{p}^*\omega_C \otimes \mathcal{A}^{\vee}) \otimes \pi^*\tilde{q}_*(\tilde{p}^*\omega_C \otimes \mathcal{A}^{\vee}) .$$}}
	Tensoring the above morphism by $\pi^*\tilde{q}_*\mathcal{A}(d-g)$ and using the natural composition
	{\small{$$\tilde{q}_*(\tilde{p}^*\omega_C \otimes \mathcal{A}^{\vee})\otimes \tilde{q}_*\mathcal{A} \to \tilde{q}_*(\tilde{p}^*\omega_C \otimes \mathcal{A}^{\vee}\otimes \mathcal{A}) \to \tilde{q}_*(\tilde{p}^*\omega_C),$$}}
	we obtain a morphism
	{\small{$$\wedge^{g-d+1}\pi^*\tilde{q}_*(\tilde{p}^*\omega_C \otimes \mathcal{A}^{\vee}) \otimes \pi^*\tilde{q}_*\mathcal{A}(d-g) \to \wedge^{g-d}\left( \pi^*\tilde{q}_*(\tilde{p}^*\omega_C \otimes \mathcal{A}^{\vee})(-1)\right)\otimes \pi^*\tilde{q}_*(\tilde{p}^*\omega_C).$$}}
	Using the natural composition
	{\small{$$\pi^*\tilde{q}_*(\tilde{p}^*\omega_C \otimes \mathcal{A}^{\vee})(-1) \hookrightarrow \pi^*\tilde{q}_*(\tilde{p}^*\omega_C \otimes \mathcal{A}^{\vee})\otimes \pi^*\tilde{q}_*\mathcal{A}\to
	\pi^*\tilde{q}_*(\tilde{p}^*\omega_C)\simeq \mathrm{H}^0(\omega_C)\otimes \mathcal{O}_{X^1_d(C)},$$}}
	we arrive at a natural morphism
	{\small{\begin{align*}
	 \wedge^{g-d+1}\pi^*\tilde{q}_*(\tilde{p}^*\omega_C \otimes \mathcal{A}^{\vee}) \otimes \pi^*\tilde{q}_*\mathcal{A}(d-g) \to \wedge^{g-d}\mathrm{H}^0(\omega_C) \otimes \mathrm{H}^0(\omega_C) \otimes \mathcal{O}_{X^1_d(C)},
	\end{align*}}}
	and composing with the quotient
	{\small{$$ \wedge^{g-d}\mathrm{H}^0(\omega_C) \otimes \mathrm{H}^0(\omega_C) \otimes \mathcal{O}_{X^1_d(C)} \to \frac{\wedge^{g-d}\mathrm{H}^0(\omega_C) \otimes \mathrm{H}^0(\omega_C)}{\wedge^{g-d+1}\mathrm{H}^0(\omega_C)} \otimes \mathcal{O}_{X^1_d(C)} $$}}
	we arrive at
	{\small{$$\delta \; : \;  \wedge^{g-d+1}\pi^*\tilde{q}_*(\tilde{p}^*\omega_C \otimes \mathcal{A}^{\vee}) \otimes \pi^*\tilde{q}_*\mathcal{A}(d-g) \to \frac{\wedge^{g-d}\mathrm{H}^0(\omega_C) \otimes \mathrm{H}^0(\omega_C)}{\wedge^{g-d+1}\mathrm{H}^0(\omega_C)} \otimes \mathcal{O}_{X^1_d(C)}.$$}}
	By the proof of Lemma \ref{natural-injection}, $\delta$ vanishes on $\wedge^{g-d+1}\pi^*\tilde{q}_*(\tilde{p}^*\omega_C \otimes \mathcal{A}^{\vee})(-1+d-g)$ and induces a morphism
	{\small{$$\delta' \; : \; \mathcal{L} \to  \frac{\wedge^{g-d}\mathrm{H}^0(\omega_C) \otimes \mathrm{H}^0(\omega_C)}{\wedge^{g-d+1}\mathrm{H}^0(\omega_C)} \otimes \mathcal{O}_{X^1_d(C)}.$$}}
	Again by Lemma \ref{natural-injection}, $\delta'$ lands in the subbundle $\frac{\mathrm{Ker}(d)}{\wedge^{g-d+1}\mathrm{H}^0(\omega_C)} \otimes \mathcal{O}_{X^1_d(C)}$, where $d$ is the differential
	{\small{$$d \; : \; \wedge^{g-d}\mathrm{H}^0(\omega_C) \otimes \mathrm{H}^0(\omega_C) \to \wedge^{g-d-1}\mathrm{H}^0(\omega_C) \otimes \mathrm{H}^0(\omega_C^{\otimes 2}).$$}}
	Hence $\delta'$ induces a morphism $i :\mathcal{L} \to \mathrm{K}_{g-d,1}(C,\omega_C) \otimes \mathcal{O}_{X^1_d(C)}$ with the required properties.
	\end{proof}
	\begin{remark} \label{O(g-d-1)}
	Note that, in the situation where $W^1_d(C)$ is zero-dimensional and reduced, then any coherent sheaf of the form $\pi^* \mathcal{A} \in \mathrm{Coh}(X^1_d(C))$ is free, for $\mathcal{A} \in \mathrm{Coh}(\widetilde{W}^1_d(C))$. We see from this that $\mathcal{L} \simeq \mathrm{det}(\mathcal{L})\simeq \mathcal{O}(d+1-g) $, where $\mathcal{L}$ is the line bundle from the above proof. Thus $$S^*\mathcal{O}_{\PP}(1)\simeq \mathcal{O}_{X^1_d(C)}(g-d-1)$$ for $\PP:=\PP(\mathrm{K}_{g-d,1}(C,\omega_C)),$ under the assumption that $W^1_d(C)$ is zero-dimensional and reduced.
	\end{remark}

We end this section with a result on the Brill--Noether variety of minimal pencils on nodal curves of even genus.
\begin{prop} \label{gen-node-BN}
	Let $D$ be an general integral curve of arithmetic genus $g=2k$ with precisely $m$ nodes for $m \leq k-1$. Then $D$ has gonality $k+1$. Further, $W^2_{k+1}(D)=\emptyset$, any closed point $[A] \in W^1_{k+1}(D)$ corresponds to a locally free sheaf $A$ on $D$, and $W^1_{k+1}(D)$ is zero-dimensional and reduced.
	\end{prop}
\begin{proof}
	We first prove $D$ has gonality $k+1$, i.e.\ there is no $d \leq k$ with $W^1_d(D) \neq \emptyset$. One could prove this using admissible covers \cite{ha-mu}, but we will use torsion-free sheaves instead. By degenerating $D$ to a rational curve, it suffices to prove there exists a rational, nodal curve $D'$ of genus $g=2k$ and gonality $k+1$. By Gieseker's Theorem, there exists such a $D'$ with the property that the Petri map is injective for all \emph{line bundles} on $D'$, \cite[Prop.\ 1.2] {gieseker}. Furthermore, the same property holds for all partial normalizations of $D'$. We will show that $D'$ has gonality $k+1$. Let $A \in W^1_d(D')$ for $d \leq k$. By Gieseker's Theorem, $A$ cannot be locally free. Thus there exists some partial normalization $\mu: \widetilde{D} \to D'$ at $n > 0$ nodes of $D'$ such that $A=\mu_* \widetilde{A}$ for some $\widetilde{A} \in W^1_{d-n}(\widetilde{D})$, with $\widetilde{A}$ locally free. As Petri's Condition holds for the line bundle $\widetilde{A}$ on the desingularization $\widetilde{D}$, we have
	$$d-n \geq \lfloor \frac{g(\widetilde{D})+3}{2} \rfloor\geq \frac{2k-n+2}{2}. $$
	Thus $d \geq \frac{2k+n+2}{2} \geq k+1$, which is a contradiction.\smallskip

	To complete the proof, by degeneration it suffices to show that $W^1_{k+1}(D')$ is zero-dimensional and reduced, that $W^2_{k+1}(D')=\emptyset$ and that each closed point of $W^1_{k+1}(D')$ corresponds to a locally free sheaf, for the rational curve $D'$ above. The fact that each closed point $[A]$ corresponds to a locally free sheaf follows from the previous paragraph. Indeed, otherwise $A=\mu_* \widetilde{A}$ for some $\widetilde{A} \in W^1_{k+1-n}(\widetilde{D})$, with $\widetilde{A}$ locally free and $\mu: \widetilde{D} \to D'$ a partial normalization at $n > 0$ nodes. As above, this implies $k+1-n \geq \frac{2k-n+2}{2}$ which is impossible for $n>0$. Since Petri's Condition holds for all line bundles on $D'$, it follows that $W^1_{k+1}(D')$ is zero-dimensional and reduced outside of $W^2_{k+1}(D')$. It remains to prove $W^2_{k+1}(D')=\emptyset$. If $[A] \in W^2_{k+1}(D')$, write $A=\mu_* \widetilde{A}$, for $\widetilde{A}$ locally free and $\mu: \widetilde{D} \to D'$ a partial normalization at $n \geq 0$ nodes (notice that we are allowing $n=0$). Since Petri's Condition holds for $\widetilde{A}$ on $\widetilde{D}$, we have $$(r+1)(g-n-(k+1-n)+r)=(r+1)(g-k+r) \leq g-n=g(\widetilde{D}),$$ where $r+1=h^0(\widetilde{D},\widetilde{A})=h^0(D,A) \geq 3$. But $$(r+1)(g-k+r)\geq 3(g-k+2)=3(k+2)>g=2k, $$
	which is a contradiction.
	\end{proof}

 \section{The Projection Map} We now recall an important construction from \cite[\S 2]{aprodu-higher}, see also \cite[\S 2.2.1]{aprodu-nagel}. Let $M$ be a graded $S_L$ module and let $f: \mathrm{H}^0(X,L) \twoheadrightarrow \C$ be a surjection. Set $W:=\mathrm{Ker}(f)$. We have the \emph{Aprodu projection map}
$$pr_{f} \; : \; \mathrm{K}_{p,1}(M, \mathrm{H}^0(L)) \to \mathrm{K}_{p-1,1}(M,W),$$
which is induced on Koszul cohomology from the map
{\small{\begin{align*}
\iota_f \; : \; \bigwedge^p \mathrm{H}^0(L) \otimes M_1 &\to \bigwedge^{p-1} W \otimes M_1  \\
v_1 \wedge \ldots \wedge v_p \otimes m &\mapsto \sum_i (-1)^i v_1 \wedge \ldots \wedge \hat{v_i} \wedge \ldots \wedge v_p \otimes f(v_i) m.
\end{align*}}}
Let $L$ be an effective line bundle on an integral curve $C$ and let $C_{sm}$ denote the smooth locus. If $x \in C_{sm}$ is not a base point of $L$, we let
$$pr_x \, :  \, \mathrm{K}_{p,1}(C,L) \to \mathrm{K}_{p-1,1}\left(\Gamma_C(L), \mathrm{H}^0(L(-x))\right)$$
denote the projection map $pr_x:=pr_{ev_x}$ associated to the evaluation map $ev_x : \mathrm{H}^0(C,L) \twoheadrightarrow \C\simeq \mathrm{H}^0(\mathcal{O}_x)$.\\

We first develop a more explicit description of the projection map $pr_x$. Let $\overline{\alpha} \in \mathrm{K}_{p,1}(C,L)$, {\small{$$\alpha \in \mathrm{Ker}\left(\bigwedge^p \mathrm{H}^0(L) \otimes \mathrm{H}^0(L) \xrightarrow{d} \bigwedge^{p-1} \mathrm{H}^0(L) \otimes \mathrm{H}^0(L^2) \right).$$}}
We have $\mathrm{H}^0(L) = \mathrm{H}^0(L(-x)) \oplus \C \langle s \rangle$ where $s \in \mathrm{H}^0(L)$ is a section with $ev_x(s)=1$, so that
{\small{$$\bigwedge^p \mathrm{H}^0(L) = \bigwedge^p \mathrm{H}^0(L(-x)) \bigoplus \left(\bigwedge^{p-1}\mathrm{H}^0(L(-x)) \right)\wedge s.$$}}
Set
{\small{\begin{align*}
V_1 &:= \bigwedge^p \mathrm{H}^0(L(-x)) \otimes \mathrm{H}^0(L(-x)) , &
V_2 &:= \bigwedge^p \mathrm{H}^0(L(-x)) \otimes s, \\
V_3 &:= \left( \bigwedge^{p-1}\mathrm{H}^0(L(-x)) \right)\wedge s \otimes \mathrm{H}^0(L(-x)), &
V_4 &:= \left( \bigwedge^{p-1}\mathrm{H}^0(L(-x)) \right)\wedge s \otimes s .
\end{align*}}}
Then $\bigwedge^p \mathrm{H}^0(L) \otimes \mathrm{H}^0(L)=V_1 \oplus V_2 \oplus V_3 \oplus V_4$. For $\alpha \in \bigwedge^p \mathrm{H}^0(L) \otimes \mathrm{H}^0(L)$, write $$\alpha = \alpha_1+\alpha_2+\alpha_3+\alpha_4, \; \text{ for $\alpha_i \in V_i$, $1 \leq i \leq 4$}.$$
We further write $\alpha_2=\alpha_2' \otimes s, \alpha_3=\alpha_3' \wedge s$ for $\alpha_2'\in \wedge^p \mathrm{H}^0(L(-x))$, $\alpha_3' \in \wedge^{p-1}\mathrm{H}^0(L(-x)) \otimes \mathrm{H}^0(L(-x))$.\\
\begin{prop} \label{image-projection}
Let $L$ be an effective line bundle on an integral curve $C$. Suppose $L$ is globally generated at $x \in C_{sm}$. 
\begin{enumerate}
\item Aprodu's projection map $pr_x$ is the composite of the map
{\small{\begin{align*}
q_x \; : \; \mathrm{K}_{p,1}(C,L) &\to \bigwedge^{p-1}\mathrm{H}^0(L(-x)) \otimes \mathrm{H}^0(L(-x))\\
\bar{\alpha} &\mapsto d\alpha_2'+(-1)^p\alpha_3'
\end{align*}}}
with the natural quotient map {\small{$$\mathrm{Ker}(d) \seq \bigwedge^{p-1}\mathrm{H}^0(L(-x)) \otimes \mathrm{H}^0(L(-x)) \to \mathrm{K}_{p-1,1}(\Gamma_C(L), \mathrm{H}^0(L(-x))).$$}} In particular, the image of $pr_x$ is contained in {\small{$$\mathrm{K}_{p-1,1}(C,L(-x)) \seq \mathrm{K}_{p-1,1}(\Gamma_C(L), \mathrm{H}^0(L(-x))).$$}}
\item If $x$ is not a base point of $L(-x)$, $q_x$ lands in $\wedge^{p-1}\mathrm{H}^0(L(-x)) \otimes \mathrm{H}^0(L(-2x)).$ In particular, $pr_x$ lands in the image of the multiplication map {\small{$$\gamma_x \, : \,\mathrm{K}_{p-1,1}(C,-x, L(-x)) \to \mathrm{K}_{p-1,1}(C,L(-x))$$}} induced from the inclusion $\Gamma_C(-x,L) \hookrightarrow \Gamma_C(L).$
\end{enumerate} 
\end{prop}
\begin{rem}
When $L$ is very ample, a different proof of the second statement in part $(1)$ of the above proposition is given in \cite[Lemma 3.1]{aprodu-higher}.
\end{rem}
\begin{proof}
\noindent \begin{enumerate}
\item Let $\alpha = \alpha_1+\alpha_2+\alpha_3+\alpha_4 \in \wedge^p \mathrm{H}^0(L) \otimes \mathrm{H}^0(L)$ for $\alpha_i \in V_i$, $1 \leq i \leq 4$ with $d \alpha=0$. Write 
$\alpha_2=\alpha_2' \otimes s$ and  $\alpha_3=\alpha_3' \wedge s$ for {\small{$$\alpha_2'\in \wedge^p \mathrm{H}^0(L(-x)), \; \; \alpha_3' \in \wedge^{p-1}\mathrm{H}^0(L(-x)) \otimes \mathrm{H}^0(L(-x)).$$ }}
Let $\beta \in \wedge^{p+1}\mathrm{H}^0(L)$. Write 
{\small{$$\beta=\beta_1+\beta_2 \wedge s, \; \;  \text{with $\beta_1\in \wedge^{p+1}\mathrm{H}^0(L(-x))$, $\beta_2 \in \wedge^p \mathrm{H}^0(L(-x))$}.$$ }}Then
$d \beta_1 \in V_1$ and $d(\beta_2 \wedge s)=d\beta_2 \wedge s+(-1)^{p+1}\beta_2 \otimes s \in V_3 \oplus V_2.$
We have the decomposition
{\small{$$\alpha+d\beta=(\alpha_1+d\beta_1)+(\alpha_2'+(-1)^{p+1}\beta_2) \otimes s+(\alpha_3'+d\beta_2) \wedge s+\alpha_4,$$}}
from which the well-definedness of $q_x$ follows.\smallskip

Next, observe that {\small{$$\iota_x(\alpha_1)=\iota_x(\alpha_2)=0,$$}} where $\iota_x \; : \wedge^p \mathrm{H}^0(L) \otimes \mathrm{H}^0(L) \to \wedge^{p-1} \mathrm{H}^0(L) \otimes \mathrm{H}^0(L)$ is defined as $\iota_x:=\iota_{ev_x}$. We have $\mathrm{H}^0(L^2)=\mathrm{H}^0(L^2(-x))\oplus \C\langle s^2 \rangle$ so that
{\small{ $$\wedge^{p-1}\mathrm{H}^0(L) \otimes \mathrm{H}^0(L^2)= \wedge^{p-1}\mathrm{H}^0(L) \otimes \mathrm{H}^0(L^2(-x)) \oplus \wedge^{p-1}\mathrm{H}^0(L) \otimes s^2.$$}}
 Let $\pi_1  :  \wedge^{p-1}\mathrm{H}^0(L) \otimes \mathrm{H}^0(L^2) \to \wedge^{p-1} \mathrm{H}^0(L) \otimes s^2$ denote the projection. Observe $$d(V_i) \seq \wedge^{p-1}\mathrm{H}^0(L) \otimes \mathrm{H}^0(L^2(-x)), \; \text{for $i=1,2,3$,}$$  and hence $\pi_1 \circ (d \alpha_i)=0$ for $i=1,2,3$. The composition $\pi_1 \circ d$ is, up to sign, the natural inclusion
{\small{ $$\wedge^{p-1} \mathrm{H}^0(L(-x)) \wedge s \otimes s \xrightarrow{\sim} \wedge^{p-1}\mathrm{H}^0(L(-x)) \otimes s^2 \seq \wedge^{p-1}\mathrm{H}^0(L) \otimes s^2.$$}}
 Since $\pi_1 \circ d(\alpha)=0$, we conclude $\alpha_4=0.$ Since $\iota_x(\alpha_1)=\iota_x(\alpha_2)=0$, we have $$\iota_x(\alpha)=\iota_x(\alpha_3)=\iota_x(\alpha_3' \wedge s)=(-1)^p \alpha'_3,$$
 since $\alpha_3' \in \wedge^{p-1}\mathrm{H}^0(L(-x)) \otimes \mathrm{H}^0(L(-x))$ and $ev_x(s)=1$.
 Hence $$q_x(\alpha)=\iota_x(\alpha)+d\alpha_2'.$$
 Since $d(\alpha)=0$ we have $d (\iota_x(\alpha))=0$, \cite[\S 2]{aprodu-higher}. It follows, firstly, that $q_x$ lands in $\mathrm{Ker}(d)$ and, secondly, $$\overline{q_x(\alpha)}=\overline{\iota_x(\alpha)}=pr_x(\bar{\alpha})$$ as required. This gives $(1)$. Note that, since the inclusion $\Gamma_C(L(-x)) \hookrightarrow \Gamma_C(L)$ of $S_{L(-x)}$ modules is an isomorphism in degree zero, we have an inclusion $$\mathrm{K}_{q,1}(C,L(-x)) \hookrightarrow \mathrm{K}_{q,1}(\Gamma_C(L), \mathrm{H}^0(L(-x)))$$ for any $q$.\smallskip
 
 \item We need to show $d\alpha_2' +(-1)^p\alpha_3'\in \wedge^{p-1}\mathrm{H}^0(L(-x)) \otimes \mathrm{H}^0(L(-2x)).$ Assuming that $L(-x)$ is globally generated at $x$, let $t \in \mathrm{H}^0(L(-x))$ be a section with $ev_x(t)=1$. Then $\mathrm{H}^0(L(-x))=\mathrm{H}^0(L(-2x))\oplus \C\langle t \rangle$. We set
{\small{
{ \begin{align*}
V_{2,1} &:=\bigwedge^p \mathrm{H}^0(L(-2x)) \otimes s, &
V_{2,2} &:=\bigwedge^{p-1} \mathrm{H}^0(L(-2x))\wedge t \otimes s, \\
V_{3,1} &:= \left( \bigwedge^{p-1}\mathrm{H}^0(L(-x)) \right)\wedge s \otimes \mathrm{H}^0(L(-2x)), &
V_{3,2} &:= \left( \bigwedge^{p-1}\mathrm{H}^0(L(-x)) \right)\wedge s \otimes t
\end{align*}}
}}
We have decompositions $V_i=V_{i,1} \oplus V_{i,2}$ for $i=2,3$. Write $\alpha_i=\alpha_{i,1}+\alpha_{i,2}$ for $\alpha_{i,1} \in V_{i,1}$, $\alpha_{i,2} \in V_{i,2}$ and $i=2,3$. Let
$$\alpha_{3,2}=\omega_1 \wedge s \otimes t, \; \; \alpha_{2,2}=\omega_2 \wedge t \otimes s$$ for
$\omega_1 \in \wedge^{p-1}\mathrm{H}^0(L(-x)), \omega_2 \in \wedge^{p-1}\mathrm{H}^0(L(-2x))$.
Since 
$$d(\omega_2 \wedge t)=d\omega_2 \wedge t+(-1)^{p}\omega_2 \otimes t,$$
and $d\omega_2 \in \wedge^{p-2}\mathrm{H}^0(L(-2x)) \otimes \mathrm{H}^0(L(-2x))$, it suffices to prove $\omega_1=-\omega_2.$\smallskip

We have  $\alpha=\alpha_1+\alpha_{2,1}+\alpha_{2,2}+\alpha_{3,1}+\alpha_{3,2},$ recalling that $\alpha_4=0$. Observe
{\small{$$d(V_1), d(V_{2,1}), d(V_{3,1}) \seq \wedge^{p-1}\mathrm{H}^0(L) \otimes \mathrm{H}^0(L^2(-2x)).$$}}
We have
{\small{$$\wedge^{p-1}\mathrm{H}^0(L) \otimes \mathrm{H}^0(L^2) =\wedge^{p-1}\mathrm{H}^0(L) \otimes \mathrm{H}^0(L^2(-2x))\oplus \wedge^{p-1}\mathrm{H}^0(L) \otimes \C\langle s^2, st \rangle.$$}}
Let $\pi_2 : \wedge^{p-1}\mathrm{H}^0(L) \otimes \mathrm{H}^0(L^2) \to  \wedge^{p-1}\mathrm{H}^0(L) \otimes \C\langle s^2, st \rangle$ be the projection.
Then 
$$0=\pi_2(d\alpha)=\pi_2(d\alpha_{2,2})+\pi_2(d\alpha_{3,2}).$$ We have
{\small{\begin{align*}
\pi_2(d(\alpha_{2,2}))&=\pi_2(d(\omega_2 \wedge t \otimes s))=(-1)^p\omega_2 \otimes st,\\
\pi_2(d(\alpha_{3,2}))&=\pi_2(d(\omega_1 \wedge s \otimes t))=(-1)^p\omega_1 \otimes st
\end{align*}}}
Thus the equation $\pi_2(d\alpha_{2,2})+\pi_2(d\alpha_{3,2})=0$ gives $\omega_1+\omega_2=0$ as required.
\end{enumerate}
\end{proof}
\smallskip

Our first goal is to offer an improvement of Proposition 3.3 from \cite{projecting}. We recall the set-up from \S 3 of \cite{projecting}. Let $C$ be an integral, nodal curve of arithmetic genus $g\geq 3$, and assume that the canonical linear system $\omega_C$ is very ample and $C \seq \PP^{g-1}$ is normally generated. Assume $C$ has precisely $m$ nodes and no other singularities. Choose general points $x,y \in C_{sm}$ in the smooth locus of $C$ and let $D$ be the nodal curve of genus $g+1$ obtained by identifying $x,y$. Then $\omega_D$ is very ample, \cite[Thm.\ 3.6]{CFHR}. Let $p \in D$ be the node over $x,y$ and let $\mu: C \to D$ be the partial normalization map.\smallskip

We embed $D$ in $\PP^g$ via the canonical linear system. Note that $D$ is normally generated. Indeed, $D$ being normally generated is equivalent to $K_{0,2}(D,\omega_D)=0$ which is in turn equivalent to $\mathrm{K}_{g+1-3,1}(D,\omega_D)=0$ by duality. This is equivalent to $\mathrm{K}_{g-2,1}(C,\omega_C(x+y))=0$ and is implied by $\mathrm{K}_{g-3,1}(C,\omega_C)=0$ by \cite[Thm.\ 3]{aprodu-higher}. This last vanishing is equivalent to $(C, \omega_C)$ being normally generated by duality.\smallskip

We now consider the projection $$\pi_p : \PP^g \dashrightarrow \PP^{g-1}$$ from the node $p$. Then $C \seq \PP^{g-1}$ is the projection $\pi_p(D)$. We let $Z \seq \PP^g$ denote the cone over $C=\pi_p(D)$ with vertex at $p$. Then $D \seq Z$. We denote by $\nu: \widetilde{Z} \to Z$ the desingularization of $Z$. The strict transform $D' \seq \widetilde{Z}$ of $D$ is isomorphic to $C$ with $\mu \sim \nu_{|_{D'}}$. From \cite[Lemma 3.2]{projecting} we have $\widetilde{Z} \simeq \PP(\mathcal{O}_C \oplus \omega_C)$ and $$\mathcal{O}_{\widetilde{Z}}(D') \simeq \mathcal{H} \otimes \iota^* \mathcal{O}_C(x+y),$$ where $\mathcal{H}$ is the pullback of $\mathcal{O}_{\PP^g}(1)$ and $\iota: \PP(\mathcal{O}_C \oplus \omega_C) \to C$ is the projection.\\

Denote by $S$ the graded rings $\mathrm{Sym}\, \mathrm{H}^0(\widetilde{Z}, \mathcal{H})$. We have a morphism
{\small{$$f \; : \; \bigoplus_q \mathrm{H}^0(\widetilde{Z}, \mathcal{H}^{\otimes q}) \to \bigoplus_q \mathrm{H}^0(C,\omega_C^{\otimes q}(qx+qy))$$}}
of graded $S$ modules, given by restriction to $D' \simeq C$. We let $\mathbb{M} \seq  \bigoplus_q \mathrm{H}^0(C,\omega_C^{\otimes q}(qx+qy))$ denote the image of $f$. Pullback induces an identification $\mathrm{H}^0(\widetilde{Z}, \mathcal{H}) \simeq \mathrm{H}^0(\mathcal{O}_Z(1))$, and so we may consider the canonical ring $\displaystyle{\Gamma_D(\omega_D):= \bigoplus_q \mathrm{H}^0(\omega^{\otimes q}_D)}$
as an $S$ module. The inclusion $\mathcal{O}_D \hookrightarrow \mu_* \mathcal{O}_C$ induces a natural inclusion {\small{$$\Gamma_D(\omega_D) \seq  \bigoplus_q \mathrm{H}^0(C,\omega_C^{\otimes q}(qx+qy)),$$}} since $\mu^* \omega_D \simeq \omega_C(x+y)$. Under this inclusion, $\Gamma_D(\omega_D)$ is the image of $\alpha$, giving the isomorphism $$\mathbb{M} \simeq \Gamma_D(\omega_D),$$ since $D \seq \PP^g$ is normally generated. As in \cite[\S 3]{projecting}, we consider the short exact sequence
{\small{$$0 \to \bigoplus_{q \in \mathbb{Z}} \mathrm{H}^0\left(\mathcal{H}^{\otimes q-1}(-\iota^*(x+y)) \right) \to \bigoplus_{q \in \mathbb{Z}} \mathrm{H}^0(\mathcal{H}^{\otimes q}) \xrightarrow{f} \Gamma_D(\omega_D) \to 0,$$}}
of $S$ modules, where we have used the isomorphism $\mathbb{M} \simeq \Gamma_D(\omega_D)$. Taking the long exact sequence of Koszul cohomology gives an exact sequence
{\small{$$0 \to \mathrm{K}_{p,1}(C,\omega_C) \to \mathrm{K}_{p,1}(D,\omega_D) \to \mathrm{K}_{p-1,1}(\widetilde{Z},-\iota^*(x+y), \mathcal{H}) \to \mathrm{K}_{p-1,2}(C,\omega_C) \to \mathrm{K}_{p-1,2}(D,\omega_D) \to \ldots$$}}
cf.\ \cite[Prop.\ 3.3]{projecting}. 
\begin{lem} \label{lemma-restriction-surj}
Embed $C$ in $\widetilde{Z}\simeq \PP(\mathcal{O}_C \oplus \omega_C)$ as a hyperplane section, (equivalently, as the section induced by the trivial quotient $\mathcal{O}_C \oplus \omega_C \twoheadrightarrow \omega_C$). Then, for any $q$ the restriction maps
$$\mathrm{H}^0(\widetilde{Z},\mathcal{H}^{\otimes q} \otimes \iota^*\mathcal{O}_C(-x-y)) \to \mathrm{H}^0(C,\omega_C^{\otimes q}(-x-y))$$
are surjective.
\end{lem}
\begin{proof}
We have $\displaystyle{\mathrm{H}^0(\widetilde{Z}, \mathcal{H}^{\otimes q} \otimes \iota^*\mathcal{O}_C(-x-y)) \simeq \mathrm{H}^0(C,\mathrm{Sym}^q(\mathcal{O}_C \oplus \omega_C)(-x-y)).}$
The trivial projection map $\mathcal{O}_C \oplus \omega_C \twoheadrightarrow \omega_C$ induces the restriction map 
$$\mathrm{H}^0(C,\mathrm{Sym}^q(\mathcal{O}_C \oplus \omega_C)(-x-y)) \to \mathrm{H}^0(C,\omega_C^{\otimes q}(-x-y)).$$
Since the projection $\mathcal{O}_C \oplus \omega_C \twoheadrightarrow \omega_C$ splits, we deduce that $\omega_C^{\otimes q}$ is a direct summand of 
$\mathrm{Sym}^q(\mathcal{O}_C \oplus \omega_C)$ and hence the restriction map above is surjective.
\end{proof}
As a consequence of the above lemma we deduce:
\begin{lem} \label{restriction-twist-x-y}
We have $\mathrm{K}_{p,q}(\widetilde{Z},-\iota^*(x+y), \mathcal{H}) \simeq \mathrm{K}_{p,q}(C,-x-y,\omega_C)$, for integers $p,q$.
\end{lem}
\begin{proof}
By Lemma \ref{lemma-restriction-surj}, we have a short exact sequence
{\small{$$0 \to \bigoplus_{q \in \mathbb{Z}} \mathrm{H}^0( \mathcal{H}^{\otimes q-1} \otimes \iota^*\mathcal{O}_C(-x-y)) \to \bigoplus_{q \in \mathbb{Z}} \mathrm{H}^0( \mathcal{H}^{\otimes q} \otimes \iota^*\mathcal{O}_C(-x-y)) \to \bigoplus_{q \in \mathbb{Z}} \mathrm{H}^0(C,\omega_C^{\otimes q}(-x-y)) \to 0.$$}}
The claim now follows from the proof of the Green--Lefschetz Theorem, see \cite[Thm.\ 3.b.7]{green-koszul} and \cite[Lemma 2.2]{generic-secant}.
\end{proof}
By Lemma \ref{restriction-twist-x-y}, we have the long exact sequence
{\small{\begin{equation} \label{les-projection}
0 \to \mathrm{K}_{p,1}(C,\omega_C) \to \mathrm{K}_{p,1}(D,\omega_D) \xrightarrow{\delta} \mathrm{K}_{p-1,1}(C,-x-y, \omega_C) \to \mathrm{K}_{p-1,2}(C,\omega_C) \to \ldots
\end{equation}}}
Let $D$ be the nodal curve of arithmetic genus $g+1$ obtained by identifying general points $x,y\in C_{sm}$ on an integral, nodal curve $C$ with $\omega_C$ very ample and $(C,\omega_C)$ normally generated. Let $\mu: C \to D$ be the partial normalization morphism. Let $$\alpha \neq 0 \in \mathrm{K}_{p,1}(D,\omega_D)$$ be a syzygy of rank $p+1$. By \cite[Corollary 5.2]{bothmer-JPAA}, \cite[Lemma 3.21]{aprodu-nagel}, the syzygy scheme $$X_{\alpha}:=\mathrm{Syz}(\alpha) \seq \PP^g$$ of $\alpha$ is a rational normal scroll of degree $p+1$ and codimension $p$ containing $D\seq \PP^g$. Letting $\widetilde{X}_{\alpha}$ denote the desingularization of $X_{\alpha}$, the Picard group of the projective bundle $\widetilde{X}_{\alpha}$ is generated by $\mathcal{O}_{X_{\alpha}}(1)$, together with the class of a ruling $R$. Assume that $D$ lies in the smooth locus of $X_{\alpha}$. As $D$ lies in the smooth locus, we may treat $D$ as a subscheme of $\widetilde{X}_{\alpha}$, and we denote the restriction of the ruling to $D$ as the line bundle $$L_{\alpha}:=\mathcal{O}_D(R).$$
Let $\displaystyle{\gamma_{x,y} \; : \; \mathrm{K}_{p,1}(C,-x-y , \omega_C) \to \mathrm{K}_{p,1}(C,\omega_C)}$ be the natural morphism induced by the inclusions $\mathrm{H}^0(\omega_C^{\otimes q}(-x-y)) \seq \mathrm{H}^0(\omega_C^{\otimes q}) $ for all $q \in \mathbb{Z}$. 
\begin{prop} \label{lin-combo-min-rk}
Let $D$ be the nodal curve of arithmetic genus $g(D):=g+1$ obtained by identifying general points $x,y\in C_{sm}$ on an integral, nodal curve $C$ of genus $g$ with $\omega_C$ very ample and $(C,\omega_C)$ normally generated. Let $\displaystyle{\alpha \neq 0 \in \mathrm{K}_{p,1}(D,\omega_D)}$ be a syzygy of least possible rank $p+1$ as above. Assume $D$ lies in the smooth locus of the syzygy scheme $X_{\alpha}$ and $h^0(C,\mu^*L_{\alpha})=h^0(D,L_{\alpha})=2$. Assume further that $\deg(L_{\alpha})=g(D)-p$. \smallskip

Then $\gamma_{x,y}(\delta(\alpha))$ is a linear combination of syzygies $\sigma$ of minimal rank $p$. Moreover, $C$ lies in the smooth locus of $X_{\sigma}$ and has associated line bundles $L_{\sigma}\simeq \mu^*L_{\alpha}$.
\end{prop}
\begin{proof}
 Let $\pi_p \; : \; \PP^g \dashrightarrow \PP^{g-1}$ be the projection from $p$. The scroll $X_{\alpha}$ coincides with the scroll induced by the pencil $|L_{\alpha}|$, by the assumptions $\deg(L_{\alpha})=g(D)-p$ and $h^0(D,L_{\alpha})=2$, \cite[Thm.\ 2]{eisenbud-harris-minimal}, \cite[\S 4]{schreyer1}. Since $L_{\alpha}$ is base-point free from the definition, the same is true for the line bundle $\mu^*L_{\alpha}$ on $C$. The projection $\pi_p(X_{\alpha})$ is the scroll induced by $\mu^*L_{\alpha}$. By the assumption $h^0(C,\mu^*L_{\alpha})=h^0(D,L_{\alpha})=2$, $\pi_p(X_{\alpha})$ is a rational normal scroll of degree $p$, \cite[\S 4]{schreyer1}. The curve $C$ lies in the smooth locus of $\pi_p(X_{\alpha})$ as $L_{\alpha}$ is base-point free. \smallskip

 We now follow \cite[\S 3]{projecting}. Let $Y_{\alpha} \supseteq X_{\alpha}$ be the cone over $\pi_p(X_{\alpha})$. We may resolve $Y_{\alpha}$ by a scroll $\widetilde{Y}_{\alpha}$, \cite{schreyer1}. Let $\mu_{Y_{\alpha}} : \widetilde{Y}_{\alpha} \to Y_{\alpha}$ be the resolution of singularities and let $X'_{\alpha}$ be the strict transform of $X_{\alpha}$. Let $\mathcal{O}_{\widetilde{Y}_{\alpha}}(1)$ be the class of a hyperplane, and let $R$ be the class of the ruling on the scroll $\widetilde{Y}_{\alpha}$. As in the proof of \cite[Thm.\ 3.6]{projecting}, pull-back gives isomorphisms $\mathrm{H}^0(\mathcal{O}_{X_{\alpha}}(q))\simeq \mathrm{H}^0(\mathcal{O}_{X'_{\alpha}}(q))$ for all $q$, so we have natural isomorphisms $\displaystyle{\mathrm{K}_{p,1}(X_{\alpha},\mathcal{O}_{X_{\alpha}}(1)) \simeq \mathrm{K}_{p,1}(X'_{\alpha},\mathcal{O}_{X'_{\alpha}}(1))}$. We further have a map
 $\displaystyle{\mathrm{K}_{p,1}(X'_{\alpha}, \mathcal{O}_{X'_{\alpha}}(1)) \xrightarrow{\Delta} \mathrm{K}_{p-1,1}(\widetilde{Y}_{\alpha}, -R; \mathcal{O}_{\widetilde{Y}_{\alpha}}(1)).}$
 Let $s \in \mathrm{H}^0(\mathcal{O}_{\widetilde{Y}_{\alpha}}(R))$ be a section such that the ruling $Z(s)$ defined by $s$ has the property $p \in \mu_{Y_{\alpha}}(Z(s))$. Then multiplication by $s$ induces a morphism
{\small{ $$\mathrm{K}_{p,1}(\widetilde{Y}_{\alpha}, -R; \mathcal{O}_{\widetilde{Y}_{\alpha}}(1)) \xrightarrow{\gamma_s} \mathrm{K}_{p,1}(\widetilde{Y}_{\alpha}, \mathcal{O}_{\widetilde{Y}_{\alpha}}(1)).$$}}
 Letting $Z$ denote the cone over $C=\pi_p(D)$ as above, note that the resolution of singularities $\widetilde{Z}=Bl_p(Z)$ naturally lies in the blow-up $Bl_p(Y_{\alpha})$ of $Y_{\alpha}$ at $p$. We have a natural, birational, morphism $g: Bl_p(Y_{\alpha}) \to \widetilde{Y}_{\alpha}$. Further, as above, let $D' \seq \widetilde{Z}$ be the strict transform of $D$, so $C \simeq D'$. Then
 $${g^*_s}_{|_{D'}} \in \mathrm{H}^0(C,\mu^*L_{\alpha}),$$
 is a section which vanishes at $x$ and $y$. As in \cite[\S 3]{projecting}, we have a commutative diagram
{\small{ $$\begin{tikzcd}
\mathrm{K}_{p,1}(X'_{\alpha}, \mathcal{O}_{X'_{\alpha}}(1)) \arrow[r, " \gamma_s \circ \Delta"] \arrow[d, "r_{D'}"] & \mathrm{K}_{p-1,1}(\widetilde{Y}_{\alpha}, \mathcal{O}_{\widetilde{Y}_{\alpha}}(1)) \arrow[d, "r_{\widetilde{Z}}"]\\
\mathrm{K}_{p,1}(D,\omega_D) \arrow[r, "\gamma_{x,y} \circ \delta"]  & \mathrm{K}_{p-1,1}(C,\omega_C).
\end{tikzcd}$$}}
where $r_{D'}$ is obtained by restricting to $D'$ and using the natural identifications $\mathrm{K}_{p,1}(D,\omega_D) \simeq \mathrm{K}_{p,1}(D',\mathcal{O}_{D'}(1))$, whereas $r_{\widetilde{Z}}$ is obtained by pulling back to $\widetilde{Z}$ and using the natural identifications $\mathrm{K}_{p-1,1}(C,\omega_C) \simeq \mathrm{K}_{p-1,1}(\widetilde{Z},\mathcal{H})$.\smallskip

By definition of the syzygy scheme $\mathrm{Syz}(\alpha)=X_{\alpha}$, there exists some $\beta \in \mathrm{K}_{p,1}(X'_{\alpha},\mathcal{O}_{X'_{\alpha}}(1)) \simeq \mathrm{K}_{p,1}(X_{\alpha},\mathcal{O}_{X_{\alpha}}(1))$ with $r_{D'}(\beta)=\alpha$. Then $\gamma_{x,y}(\delta(\alpha)) \in \mathrm{Im}(r_{\widetilde{Z}})$. By Proposition \ref{rank-scroll-2}, $\mathrm{K}_{p-1,1}(\widetilde{Y}_{\alpha}, \mathcal{O}_{\widetilde{Y}_{\alpha}}(1))$ is generated by syzygies $\sigma'$ of rank $p$. The pull-back map $r_{\widetilde{Z}}$ does not change the rank of a syzygy. Thus $\gamma_{x,y}(\delta(\alpha))$ is a linear combination of syzygies $\sigma:=r_{\widetilde{Z}}(\sigma')$ of rank $p$. We have isomorphisms
{\small{\begin{align*}
\mathrm{K}_{p-1,1}(\widetilde{Y}_{\alpha}, \mathcal{O}_{\widetilde{Y}_{\alpha}}(1)) \simeq \mathrm{K}_{p-1,1}(Y_{\alpha}, \mathcal{O}_{{\PP}^g_{|_{Y_{\alpha}}}}(1)) \simeq  \mathrm{K}_{p-1,1}(\pi_p(X_{\alpha}), \mathcal{O}_{{\PP}^{g-1}_{|_{\pi_p(X_{\alpha})}}}(1)) .
\end{align*}}}
Under this identification, $r_{\widetilde{Z}}$ becomes the natural inclusion 
{\small{$$ \mathrm{K}_{p-1,1}(\pi_p(X_{\alpha}), \mathcal{O}_{\pi_p(X_{\alpha})}(1)) \hookrightarrow \mathrm{K}_{p-1,1}(C,\omega_C) .$$}}
Thus $\pi_p(X_{\alpha}) \seq \mathrm{Syz}(\sigma)$ by definition of the syzygy scheme, \cite[Remark 3.6]{aprodu-nagel}. By \cite[Corollary 5.2]{bothmer-JPAA}, the syzygy scheme $\mathrm{Syz}(\sigma)$ is a rational normal scroll of degree $p$ and codimension $p-1$, and therefore $\mathrm{Syz}(\sigma)= \pi_p(X_{\alpha}) $, as required.
 \end{proof}

\begin{lem} \label{twisted-van}
Let $C$ be a general, integral $m$-nodal curve of arithmetic genus $g\geq 2$ for any $0 \leq m \leq g$. Let $x,y \in C$ be general points in the smooth locus of $C$. For $n \leq \lfloor \frac{g}{2} \rfloor-2$, we have $$\mathrm{K}_{n,2}(C,-x-y, \omega_C)=0.$$
Further, if $n \neq g-2$ then $\mathrm{K}_{n,3}(C,-x-y, \omega_C)=0$.
\end{lem}
\begin{proof}
By Equation (\ref{les-projection}) we have the exact sequence
{\small{$$\to \mathrm{K}_{n+1,2}(D,\omega_D) \to \mathrm{K}_{n,2}(C,-x-y, \omega_C) \to \mathrm{K}_{n,3}(D,\omega_D) \to$$}}
But $ \mathrm{K}_{p,3}(D,\omega_D)=0$ unless $p=g-1$ for the canonical curve $(D,\omega_D)$, whereas $\mathrm{K}_{n+1,2}(D,\omega_D)=0$ for $n \leq \lfloor \frac{g}{2} \rfloor-2$ by Voisin's Theorem \cite{V2}, which holds for a general $m$-nodal curve as there exist $j$-nodal curves in the primitive linear system on a general K3 surface of genus $\ell$ for all $0 \leq j \leq \ell$, \cite{chen-every-rational}. From the exact sequence
{\small{$$\to \mathrm{K}_{n+1,3}(D,\omega_D) \to \mathrm{K}_{n,3}(C,-x-y, \omega_C) \to \mathrm{K}_{n,4}(D,\omega_D) \to$$}}
and the vanishings $\mathrm{K}_{n,4}(D,\omega_D)$, valid for all $n$, $ \mathrm{K}_{n+1,3}(D,\omega_D)=0$, valid for $n+1 \neq (g+1)-2$, we see $\mathrm{K}_{n,3}(C,-x-y, \omega_C)=0$.

\end{proof}

Since $\mu_* \omega_C \simeq \omega_D \otimes I_{p}$ we have an exact sequence
$\displaystyle{0 \to \mathrm{H}^0(C,\omega_C) \to \mathrm{H}^0(D,\omega_D) \xrightarrow{ev_p} \C \to 0.}$
We have the Aprodu projection map $pr:=pr_{ev_p}$ acting on $\mathrm{K}_{p,1}(D,\omega_D)$. By \cite{aprodu-higher}, Lemma 3.1, the image of $pr$ lands in $\mathrm{K}_{p-1,1}(C,\omega_C)$ since $\pi_p(D)=C \seq \PP^{g-1}$, where $\pi_p : \PP^{g} \dashrightarrow \PP^{g-1}$ is the projection away from $p \in D$, i.e.\ we have the map
{\small{$$pr \, : \, \mathrm{K}_{p,1}(D,\omega_D) \to \mathrm{K}_{p-1,1}(C,\omega_C).$$}} 
\begin{lem} \label{im-pr-nodal}
Let $C$ be an integral $m$-nodal curve of arithmetic genus $g$ for any $0 \leq m \leq g$ such that $(C,\omega_C)$ is normally generated. Let $D$ be the $m+1$ nodal curve obtained by identifying general points $x, y$ in the smooth locus of $C$. For any $p \geq 0$, the image of $pr : \mathrm{K}_{p,1}(D,\omega_D) \to \mathrm{K}_{p-1,1}(C,\omega_C)$ is contained in 
$$\gamma_{x,y}(\mathrm{K}_{p-1,1}(C,-x-y , \omega_C))\seq \mathrm{K}_{p-1,1}(C,\omega_C).$$
\end{lem}
\begin{proof}
We have  isomorphisms $\mathrm{H}^0(D,\omega_D)\simeq \mathrm{H}^0(C,\omega_C(x+y)), \;  \mathrm{H}^0(C,\omega_C) \simeq \mathrm{H}^0(C,\omega_C(x)),$
inducing natural isomorphisms
{\small{$$\mathrm{K}_{p,1}(D,\omega_D) \simeq \mathrm{K}_{p,1}(C,\omega_C(x+y)), \; \; \mathrm{K}_{p-1,1}(C,\omega_C) \simeq \mathrm{K}_{p-1,1}(C,\omega_C(x)).$$}}
We may then identify $pr : \mathrm{K}_{p,1}(D,\omega_D) \to \mathrm{K}_{p-1,1}(C,\omega_C)$ with $pr_y: \mathrm{K}_{p,1}(C,\omega_C(x+y)) \to \mathrm{K}_{p-1,1}(C,\omega_C(x))$. Precisely, let $ s \in \mathrm{H}^0(\omega_D)$ be a section with $ev_p(s)=1$, which corresponds to a section $s'\in \mathrm{H}^0(\omega_C(x+y))$ with $ev_x(s')=ev_y(s')=1$, identifying $\mathcal{O}_x$ and $\mathcal{O}_y$ with $\mathcal{O}_p$ via $\mu$. The decomposition
{\small{\begin{align*}
\wedge^p \mathrm{H}^0(\omega_C(x+y)) \otimes \mathrm{H}^0(\omega_C(x+y)) &\simeq \wedge^p \mathrm{H}^0(\omega_C(x)) \otimes \mathrm{H}^0(\omega_C(x))\oplus
\wedge^p\mathrm{H}^0(\omega_C(x)) \otimes s' \\
&\oplus \wedge^{p-1} \mathrm{H}^0(\omega_C(x))\wedge s' \otimes \mathrm{H}^0(\omega_C(x))\oplus
\wedge^{p-1}\mathrm{H}^0(\omega_C(x))\wedge s' \otimes s'
\end{align*}}}
corresponds to the decomposition
{\small{$$\wedge^p\mathrm{H}^0(\omega_D) \otimes \mathrm{H}^0(\omega_D) \simeq V_1 \oplus V_2 \oplus V_3 \oplus V_4$$}}
with
{\small{\begin{align*}
V_1 &:= \bigwedge^p \mathrm{H}^0(\omega_C) \otimes \mathrm{H}^0(\omega_C) , &
V_2 &:= \bigwedge^p \mathrm{H}^0(\omega_C) \otimes s, \\
V_3 &:= \left( \bigwedge^{p-1}\mathrm{H}^0(\omega_C) \right)\wedge s \otimes \mathrm{H}^0(\omega_C), &
V_4 &:= \left( \bigwedge^{p-1}\mathrm{H}^0(\omega_C) \right)\wedge s \otimes s .
\end{align*}}}
Applying Proposition \ref{image-projection} to $pr_y$ and using the above identifications we see
{\small{$$pr(\bar{\alpha})=\overline{(-1)^p \alpha'_3+d\alpha'_2}$$}}
for $\alpha=\alpha_1+ \alpha'_2 \otimes s+\alpha'_3 \wedge s+\alpha_4$, with
$\alpha_1 \in V_1$, $\alpha'_2 \otimes s \in V_2$, $\alpha'_3 \wedge s \in V_3$, $\alpha_4 \in V_4$. Further, by the second part of Proposition \ref{image-projection}, applied to $pr_y$ again, 
{\small{$$(-1)^p \alpha'_3+d\alpha'_2 \in \wedge^{p-1}\mathrm{H}^0(\omega_C) \otimes \mathrm{H}^0(\omega_C(-y)) \simeq \wedge^{p-1}\mathrm{H}^0(\omega_C(x)) \otimes \mathrm{H}^0(\omega_C(x-y)).$$}}
Swapping the roles of $x$ and $y$, we see 
{\small{$$(-1)^p \alpha'_3+d\alpha'_2 \in \wedge^{p-1}\mathrm{H}^0(\omega_C) \otimes \mathrm{H}^0(\omega_C(-x)) $$}}
so that $(-1)^p \alpha'_3+d\alpha'_2 \in \wedge^{p-1}\mathrm{H}^0(\omega_C) \otimes \mathrm{H}^0(\omega_C(-x-y)), $
which gives the claim.

\end{proof}

We have a short exact sequence
{\small{\begin{equation} \label{weird-module-T}
0 \to \Gamma_C(-x-y, \omega_C) \to \Gamma_C(\omega_C) \to \mathbb{T}_{x,y}(C) \to 0
\end{equation}}}
for some graded $S$-module $\mathbb{T}_{x,y}(C)$.
\begin{prop} \label{spanned-im-gamma}
Let $C$ be a general, integral $m$-nodal curve of arithmetic genus $g\geq 2$ for any $0 \leq m \leq g$. Let $x,y \in C$ be general points in the smooth locus of $C$. Fix $p \leq \lfloor \frac{g-3}{2} \rfloor $. As the points $x,y \in C$ vary over $C$, the subspaces $\mathrm{Im}(\gamma_{x,y}) \seq \mathrm{K}_{p,1}(C, \omega_C)$ span $\mathrm{K}_{p,1}(C, \omega_C)$.
\end{prop}
\begin{proof}
In fact, we will prove that if $x,y, s, t \in C$ are general then  $\mathrm{Im}(\gamma_{x,y}) \cup \mathrm{Im}(\gamma_{s,t})$ spans $\mathrm{K}_{p,1}(C, \omega_C)$. From the long exact sequence of Koszul cohomology associated to the short exact sequence \ref{weird-module-T}, we have an exact sequence
{\small{$$\mathrm{K}_{p,1}(C,-x-y, \omega_C) \xrightarrow{\gamma_{x,y}} \mathrm{K}_{p,1}(C,\omega_C) \xrightarrow{\alpha} \mathrm{K}_{p,1}( \mathbb{T}_{x,y}(C), \mathrm{H}^0(\omega_C)) \to 0,$$}}
since $\mathrm{K}_{p-1,2}(C,-x-y, \omega_C)=0$ by Lemma \ref{twisted-van}. Thus, to prove the claim it suffices to show
{\small{$$\alpha_{|_{ \mathrm{Im}(\gamma_{s,t})}} \; : \;  \mathrm{Im}(\gamma_{s,t}) \to \mathrm{K}_{p,1}( \mathbb{T}_{x,y}(C), \mathrm{H}^0(\omega_C))$$}}
is surjective.\smallskip

Let $D'$ be the nodal curve of genus $g+1$ obtained by identifying $s$ and $t$, with partial normalization map $\mu': C \to D'$. We have a short exact sequence
{\small{$$0 \to \Gamma_{D'}(-x-y, \omega_{D'}) \to \Gamma_{D'}(\omega_{D'}) \to \mathbb{T}_{x,y}(D') \to 0$$}}
of $\overline{S}:=\mathrm{Sym} \left(\mathrm{H}^0(D',\omega_{D'}) \right)$-modules. By Lemma \ref{twisted-van} we have a surjection
{\small{$$\alpha' \; : \; \mathrm{K}_{p+1,1}(D',\omega_{D'}) \twoheadrightarrow \mathrm{K}_{p+1,1}(\mathbb{T}_{x,y}(D'),\mathrm{H}^0(\omega_{D'}))$$}}
since $p \leq \lfloor \frac{g+1}{2}\rfloor-2=\lfloor \frac{g-3}{2} \rfloor$. We have a short exact sequence
{\small{$$0 \to \mathrm{H}^0(C,\omega_C) \xrightarrow{\mu'_*} \mathrm{H}^0(D,\omega_D) \to \C \to 0,$$}}
inducing an inclusion $S \hookrightarrow \overline{S}$ of graded ring. By restriction of scalars, we may consider $\mathbb{T}_{x,y}(D')$ as an $S$-module, denoted $(\mathbb{T}_{x,y}(D'))_{S}$. Then we may identify $(\mathbb{T}_{x,y}(D'))_{S}$ with $ \mathbb{T}_{x,y}(C)$. We have a projection map 
{\small{$$\mathrm{pr} \, : \, \mathrm{K}_{p+1,1}(\mathbb{T}_{x,y}(D'), \mathrm{H}^0(\omega_{D'})) \to \mathrm{K}_{p,1}( \mathbb{T}_{x,y}(C), \mathrm{H}^0(\omega_C)),$$}}
by \cite[\S 2]{aprodu-higher}, which fits into an exact sequence
{\small{$$\to \mathrm{K}_{p+1,1}(\mathbb{T}_{x,y}(D'), \mathrm{H}^0(\omega_{D'})) \xrightarrow{pr} \mathrm{K}_{p,1}( \mathbb{T}_{x,y}(C), \mathrm{H}^0(\omega_C)) \to \mathrm{K}_{p,2}( \mathbb{T}_{x,y}(C), \mathrm{H}^0(\omega_C)) \to \ldots.$$}}
We have the exact sequence
{\small{$$\to \mathrm{K}_{p,2}(C,\omega_C) \to \mathrm{K}_{p,2}(\mathbb{T}_{x,y}(C),\mathrm{H}^0(\omega_C)) \to \mathrm{K}_{p-1,3}(C,-x-y, \omega_C) \to \ldots$$}}
We have $\mathrm{K}_{p,2}(C,\omega_C)=0$ by Voisin's Theorem, whereas $\mathrm{K}_{p-1,3}(C,-x-y, \omega_C)=0$ by Lemma \ref{twisted-van}. Thus $\mathrm{K}_{p,2}(\mathbb{T}_{x,y}(C),\mathrm{H}^0(\omega_C))=0$ and we have a \emph{surjective} map
{\small{$$\mathrm{pr} \, : \, \mathrm{K}_{p+1,1}(\mathbb{T}_{x,y}(D'), \mathrm{H}^0(\omega_{D'})) \twoheadrightarrow \mathrm{K}_{p,1}( \mathbb{T}_{x,y}(C), \mathrm{H}^0(\omega_C)).$$}}
We have a commutative diagram
{\small{$$\begin{tikzcd}
\mathrm{K}_{p+1,1}(D',\omega_{D'}) \arrow[r, two heads, "\alpha' "] \arrow[d, two heads, "pr"] & \mathrm{K}_{p+1,1}(\mathbb{T}_{x,y}(D'), \mathrm{H}^0(\omega_{D'})) \arrow[d, two heads, "pr"]\\
\mathrm{K}_{p,1}(C,\omega_C) \arrow[r, "\alpha"]  & \mathrm{K}_{p,1}( \mathbb{T}_{x,y}(C), \mathrm{H}^0(\omega_C)),
\end{tikzcd}$$}}
by functoriality of projection maps, \cite[\S 2]{aprodu-higher}. Since
{\small{$$pr(\mathrm{K}_{p+1,1}(D',\omega_{D'})) \seq \mathrm{Im}(\gamma_{s,t})\seq \mathrm{K}_{p,1}(C,\omega_C)$$}}
by Lemma \ref{im-pr-nodal}, we see that $\alpha_{|_{ \mathrm{Im}(\gamma_{s,t})}}$ is surjective, as required.
\end{proof}

\section{The Geometric Syzygy Conjecture} \label{final-section}
We will prove the Geometric Syzygy Conjecture by induction on genus. The next lemma provides the induction step.
\begin{prop} \label{induction-step}
Let $C$ be a general, integral $m$-nodal curve of arithmetic genus $g\geq 2$ for any $0 \leq m \leq g$. Let $x,y \in C_{sm}$ be general and let $D$ be the nodal curve obtained by identifying $x$ and $y$. Let $p \leq \lfloor \frac{g-1}{2} \rfloor$. Suppose $\mathrm{K}_{p,1}(D,\omega_D)$ is spanned by syzygies $\alpha$ satisfying the following properties:
\begin{enumerate} [label=(\roman*)] 
\item $\alpha$ has minimal rank $p+1$.
\item The curve $D$ lies in the smooth locus of the syzygy scheme $X_{\alpha}:=\mathrm{Syz}(\alpha)$.
\item The associated line bundle $L_{\alpha}:=\mathcal{O}_D(R)$ satisfies $\deg(L_{\alpha})=g+1-p$ and $h^0(C,\mu^*L_{\alpha})=2$, where $\mu: C \to D$ is the partial normalization map.
\end{enumerate}
Then $\mathrm{K}_{p-1,1}(C,\omega_C)$ is spanned by syzygies $\sigma$ of rank $p$ such that $X_{\sigma}:=\mathrm{Syz}(\sigma)$ is a scroll with $C$ lying in the smooth locus of $X_{\sigma}$ and with associated line bundle of the form $L_{\sigma}=\mu^*L_{\alpha}$, where $L_{\alpha}$ is associated to a syzygy $\alpha \in \mathrm{K}_{p,1}(D,\omega_D)$ of rank $p+1$.
\end{prop}
\begin{proof}
By Proposition \ref{spanned-im-gamma}, $\mathrm{K}_{p-1,1}(C,\omega_C)$ is spanned by the subspaces $\mathrm{Im}(\gamma_{x,y}) \seq \mathrm{K}_{p-1,1}(C, \omega_C)$ as $x, y \in C_{sm}$, vary. Thus it suffices to show $\mathrm{Im}(\gamma_{x,y})$ is spanned by syzygies $\sigma$ of rank $p$ and with the required properties. We have $\mathrm{K}_{p-1,2}(C,\omega_C)=0$ by generic Green's Conjecture \cite{V2}, and thus a surjection
$$\delta \; : \; \mathrm{K}_{p,1}(D,\omega_D) \twoheadrightarrow   \mathrm{K}_{p-1,1}(C,-x-y, \omega_C).$$
Since $\mathrm{K}_{p,1}(D,\omega_D)$ is spanned by syzygies $\alpha$ satisfying the properties $(i),(ii),(iii)$, it suffices to show that, for such a syzygy $\alpha$, the syzygy $\gamma_{x,y}(\delta(\alpha))$ is a linear combination of syzygies of rank $p$ with the required properties. This follows from Proposition \ref{lin-combo-min-rk}.
\end{proof}

Let $(X,L)$ be a primitively polarized K3 surface of genus $g$, i.e.\ $(L)^2=2g-2$. 
\begin{lem} \label{E}
	Let $(X,L)$ be a very general, primitively polarized, K3 surface of even genus $g=2k$ for $g \geq 2$. Then any curve $C \in |L|$ is integral, with $\omega_C$ very ample, and further $(C,\omega_C)$ is normally generated. We have $W^1_{d}(C)=\emptyset$ for $ d \leq k$ and, further, each $[A] \in W^1_{k+1}(C)$ is globally generated and has $h^0(A)=2$. Consider the bundle $E$ whose dual fits into the exact sequence
	{\small{$$0 \to {E}^{\vee} \to \mathrm{H}^0(A) \otimes \mathcal{O}_X \to A_{|_C} \to 0.$$}} Then $E$ is simple, has invariants $\mathrm{det}(E)\simeq L$, $h^0(E)=k+2$, $h^1(E)=h^2(E)=0$ and does not depend on the choice of $[A] \in W^1_{k+1}(C)$ or of $C \in |L|$. 
	\end{lem}
	\begin{proof}
	A very general polarized K3 has $\mathrm{Pic}(X)=\mathbb{Z}[L]$, which implies any $C \in |L|$ is integral and also that $L$ is very ample, \cite[Thm.\ 1.1]{knutsen}. For any $C \in |L|$, restricting the embedding $\phi_L : X \to \PP^{g}$ to the hyperplane $C$ gives the canonical embedding $\phi_{\omega_C}: C \to \PP^{g-1}$, and so $\omega_C$ is very ample. To show $(C,\omega_C)$ is normally generated, it is equivalent to show $K_{0,q}(C,\omega_C)=0$ for $q \geq 2$. This is equivalent to $K_{0,q}(X,L)=0$ for $q \geq 2$ by the Hyperplane Restriction Theorem, \cite{green-koszul}. The claim is now a (very special) case of Voisin's Theorem, \cite{V1}. \smallskip
		
		 As $C \in |L|$ is integral and lies on a smooth surface, we may construct the Brill--Noether loci $W^r_d(C) \seq \bar{J}^d(C)$. Let $[A] \in W^1_{k+1}(C)$ be globally generated with $h^0(A)=2$, and let $E^{\vee}$ be defined by the exact sequence $0 \to {E}^{\vee} \to \mathrm{H}^0(A) \otimes \mathcal{O}_X \to A_{|_C} \to 0.$ By \cite{lazarsfeld-BNP}, generalized in \cite[Lemma 2.2]{gomez} to the torsion-free case, $E^{\vee}$ as defined in the lemma is a vector bundle of rank two. As in \cite[\S 1]{lazarsfeld-BNP} the condition $\mathrm{Pic}(X)=\mathbb{Z}[L]$ (satisfied for a very general polarized K3), implies that $E$ is simple (also compare with \cite[Lemma 5.5]{kemeny-moduli} and \cite[Thm.\ 3.1]{ciliberto-knutsen}). The bundle $E$ has invariants $\mathrm{det}(E)\simeq L$, $h^0(E)=k+2$, $h^1(E)=h^2(E)=0$ as in \cite[\S 1]{lazarsfeld-BNP}. Since the moduli space of stable bundles on K3 surfaces with these invariants is zero dimensional, the bundle $E$ does not depend on the choice of $[A] \in W^1_{k+1}(C)$ or of $C \in |L|$. \smallskip
		
		We now show $[A] \in W^1_{k+1}(C)$ is globally generated and has $h^0(A)=2$. Suppose $h^0(A)\geq 3$. Let $A' \seq A$ be the subsheaf of $A$ generated by the global sections of $A$, i.e.\ $A'$ is the image of the evaluation morphism
	{\small{$$\mathrm{H}^0(A) \otimes \mathcal{O}_C \to A.$$}}
	Then $A'$ is a rank-one torsion free sheaf with $\mathrm{H}^0(A')=\mathrm{H}^0(A)$ which is generated by its global sections. We have $[A']\in W^2_{d}(C)$ for $d \leq k+1$. Let $i: C \hookrightarrow X$ be the natural inclusion. We have a Lazarsfeld--Mukai bundle $F$ on $X$, whose dual fits into the exact sequence
	{\small{$$0 \to {F}^{\vee} \to \mathrm{H}^0(A') \otimes \mathcal{O}_X \to i_*A' \to 0.$$}}
	As above, the generality of $X$ implies that $F$ is simple. This in turn implies $\rho(A')=g-h^0(A')(g-d+h^0(A')-1) \geq 0$ by \cite[\S 1]{lazarsfeld-BNP}. Thus $g-d+h^0(A')-1\leq \frac{g}{h^0(A')}\leq \frac{2k}{3}$. But $g-d+h^0(A')-1\geq 2k-(k+1)+3-1=k+1$, which is a contradiction. Similarly, if $A \in W^1_{k+1}(C)$ were not base-point free, the subsheaf $A' \seq A$ generated by the global sections of $A$ would have $\rho(A')<0$, where $\rho(A'):=g-h^0(A')h^1(A')$ is the Brill--Noether number, which cannot happen by \cite[\S 1]{lazarsfeld-BNP} as above. The same argument show that $W^1_{d}(C)=\emptyset$ for $ d \leq k$, since otherwise we would have torsion-free sheaves with negative Brill--Noether number $\rho$.
		
	\end{proof}
	We now give a description of the sections of the vector bundle $E$ from  Lemma \ref{E}. Note that we are not assuming $A$ is invertible in the lemma below.
\begin{lem} \label{E2}
Let $E$ be the vector bundle from Lemma \ref{E}. We have a natural isomorphism $$\mathrm{H}^0(E) \simeq \mathrm{Ker}\mu,$$ where $\mu  :  \mathrm{H}^0(A) \otimes \mathrm{H}^0(L) \to \mathrm{H}^0(A\otimes \omega_C),$ is the multiplication map composed with restriction.
	
	We further have an isomorphism $$\alpha: \mathrm{H}^0(A) \oplus \mathrm{H}^0(A^{\vee} \otimes \omega_C) \simeq \mathrm{Ker}\mu.$$ Explicitly, if $t \in \mathrm{H}^0(L)$ defines $C$, if $q: \mathrm{H}^0(\omega_C) \to \mathrm{H}^0(L)$ is a section of the restriction map $\mathrm{H}^0(L) \to \mathrm{H}^0(\omega_C)$ and if $\{x, y\}$ is a basis for $\mathrm{H}^0(A)$, we define $\alpha$ by the formula
	{\small{\begin{align*}
	\mathrm{H}^0(A) \oplus \mathrm{H}^0(A^{\vee} \otimes \omega_C) &\to \mathrm{Ker}\mu \\
	(v,w) &\mapsto (v\otimes t, y \otimes q(xw)-x\otimes q(yw))
	\end{align*}}}
	\end{lem}
	\begin{proof}
Twisting the defining sequence for $E^{\vee}$ by $\mathcal{O}_X(C)$ and using the natural isomorphism $E \simeq E^{\vee}(C)$, we have a short exact sequence
	{\small{$$0 \to {E} \to \mathrm{H}^0(A) \otimes \mathcal{O}_X(C) \to A\otimes \omega_C \to 0.$$ }}
	Thus $\mathrm{H}^0(E) \simeq \mathrm{Ker}\mu,$ where $\mu  :  \mathrm{H}^0(A) \otimes \mathrm{H}^0(L) \to \mathrm{H}^0(A\otimes \omega_C),$ is the multiplication map.\smallskip
	
	By Lemma \ref{tf-bpf}  we have an exact sequence
		{\small{$$0 \to A^{\vee}\otimes \omega_C \xrightarrow{i} \mathrm{H}^0(A) \otimes \omega_C \xrightarrow{ev} A  \otimes \omega_C \to 0,$$}}
		where $i(s)=y\otimes \rho(xs)-x\otimes \rho(ys),$ for $\{x, y\}$ a basis for $\mathrm{H}^0(A)$. Note that we are simply defining $A^{\vee}:=\mathcal{H}\text{om}(A,\mathcal{O}_C)$; we have $(A^{\vee})^{\vee}\simeq A$ as $A$ is reflexive. We have a commutative diagram
{\small{$$\begin{tikzcd}
	 	 & 0 \arrow[d] & 0 \arrow[d] & &\\
 &\mathrm{H}^0(A) \otimes \mathcal{O}_X \arrow[r, "\sim"] \arrow[d] & \mathrm{H}^0(A) \otimes \mathcal{O}_X \arrow[d, "\otimes t"] & &\\
	0 \arrow[r] & E \arrow[r] \arrow{d}& \mathrm{H}^0(A) \otimes L \arrow[r] \arrow[d] &A\otimes \omega_C \arrow[r] \arrow[d, "\simeq"] & 0\\
	0 \arrow[r] &  A^{\vee}\otimes \omega_C \arrow[r, "i"] \arrow[d] &  \mathrm{H}^0(A) \otimes \omega_C \arrow[r] \arrow[d] &A\otimes \omega_C \arrow[r] & 0\\
		 	 & 0  & 0 &&
\end{tikzcd}$$}}
with exact rows and columns. The claimed isomorphism from $\mathrm{H}^0(A) \oplus \mathrm{H}^0(A^{\vee} \otimes \omega_C)$ to $\mathrm{Ker}\mu \simeq \mathrm{H}^0(E)$ follows.
\end{proof}
	\begin{remark} \label{rem-EC}
	From the first column of the diagram in the above proof, we have an exact sequence
	{\small{$$0 \to \mathrm{H}^0(A) \otimes \mathcal{O}_X \to E \to A^{\vee} \otimes \omega_C \to 0.$$}} Restricting to $C$, and using the isomorphism
	{\small{$$\mathrm{Ker}(\mathrm{H}^0(A) \otimes \mathcal{O}_X \to A \otimes \mathcal{O}_C)=:E^{\vee}\simeq E(-C),$$}}
we also obtain an exact sequence $0 \to A \to E_{|_C} \to A^{\vee} \otimes \omega_C,$ of sheaves on $C$.	
	\end{remark}

	Let $E$ be the rank two Lazarsfeld--Mukai bundle from Lemma \ref{E} and consider the restriction $E_C$ to an integral curve $C \in |L|$. We have the natural determinant map
	$$\mathrm{det} \; : \; \bigwedge^2 \mathrm{H}^0(E_C) \to \mathrm{H}^0(\omega_C),$$
	obtained as the composition $\wedge^2 \mathrm{H}^0(E_C)\to \mathrm{H}^0( \wedge^2 E_C)\simeq \mathrm{H}^0(\omega_C)$.
	\begin{lem} \label{det-petri}
	Choose any section $\mathrm{H}^0(\omega_C \otimes A^{\vee}) \hookrightarrow \mathrm{H}^0(E_C)$ to the natural map $\mathrm{H}^0(E_C) \twoheadrightarrow \mathrm{H}^0(\omega_C \otimes A^{\vee}) $ from Lemma \ref{E}, inducing a splitting
	{\small{$$\mathrm{H}^0(E_C) \simeq \mathrm{H}^0(A) \oplus \mathrm{H}^0(\omega_C \otimes A^{\vee}).$$}}
	Consider the inclusion $\mathrm{H}^0(A)\otimes \mathrm{H}^0(\omega_C \otimes A^{\vee}) \seq \bigwedge^2 \mathrm{H}^0(E_C)$ resulting from the above splitting. Then 
	the restriction {\small{$$\mathrm{det}_{|_{\mathrm{H}^0(A)\otimes \mathrm{H}^0(\omega_C \otimes A^{\vee})}}$$}} of the determinant map to this subspace coincides with the Petri map 
	{\small{$$\mathrm{H}^0(A)\otimes \mathrm{H}^0(\omega_C \otimes A^{\vee}) \to \mathrm{H}^0(\omega_C).$$}}
	Moreover, the determinant map is zero on all other components:
	{\small{$$\mathrm{det}_{|_{\bigwedge^2 \mathrm{H}^0(A)}}=0, \; \; \mathrm{det}_{|_{\bigwedge^2 \mathrm{H}^0(\omega_C \otimes A^{\vee})}}=0.$$}}
	\end{lem}
	\begin{proof}
	 Restriction induces an isomorphism $\mathrm{H}^0(E) \simeq \mathrm{H}^0(E_C)$ since $\mathrm{H}^0(E(-C))=\mathrm{H}^0(E^{\vee})=0$. Let $\det_X: \wedge^2 \mathrm{H}^0(E) \to \mathrm{H}^0(L)$ denote the determinant map (on $X$). We have the commutative diagram
	{\small{$$\begin{tikzcd}
	\bigwedge^2 \mathrm{H}^0(E) \arrow[hookrightarrow]{r} \arrow[d, "\mathrm{det}_X"] & \bigwedge^2 \mathrm{H}^0(A) \otimes \mathrm{H}^0(L) \arrow[d, "\phi"] &\\
	\mathrm{H}^0(L) \arrow[hookrightarrow]{r}{j} & \mathrm{H}^0(L^{\otimes 2}) &
	\end{tikzcd}$$}}
	where $\phi$ comes from the isomorphism
	{\small{$$\wedge^2 (\mathrm{H}^0(A) \otimes L) \simeq (\wedge^2 \mathrm{H}^0(A)) \otimes L^{\otimes 2} \simeq L^{\otimes 2},$$}}
	where the last isomorphism depends on a choice of ordered basis $\{x, y\}$ of $\mathrm{H}^0(A)$, with $x \wedge y$ the generator for $\wedge^2 \mathrm{H}^0(A)$. Further, we have the formula $\mathrm{det}=(\det_X)_{|_C} :\wedge^2 \mathrm{H}^0(E_C) \to \mathrm{H}^0(\omega_C).$\smallskip
	
	We first show $\mathrm{det}_{|_{\bigwedge^2 \mathrm{H}^0(\omega_C \otimes A^{\vee})}}=0$. From Lemma \ref{E2}, it suffices to show 
	{\small{$$\phi\left((y \otimes q(xw_1)-x \otimes q(yw_1)) \wedge (y \otimes q(xw_2)-x \otimes q(yw_2))\right)=0,$$}}
	where $w_1, w_2 \in \mathrm{H}^0(\omega_C \otimes A^{\vee})$. In other words, we need
	{\small{$$\alpha=q(xw_1)q(yw_2)-q(yw_1)q(xw_2)=0 \in \mathrm{H}^0(L^{\otimes 2}).$$}}
	But $\alpha=\psi(\alpha')$ lies in the image of $\psi : \mathrm{Sym}^2\mathrm{H}^0(L) \to \mathrm{H}^0(L^{\otimes 2})$. Cearly the restriction $\alpha'_{|_C}$ of $\alpha'$ to $C$ lies in the kernel of $\psi_{|_C} : \mathrm{Sym}^2\mathrm{H}^0(\omega_C) \to \mathrm{H}^0(\omega_C^{\otimes 2})$. Since restriction induces an isomorphism
	{\small{$$\mathrm{Ker}(\psi)=\mathrm{K}_{1,1}(X,L)\simeq \mathrm{K}_{1,1}(C,\omega_C)=\mathrm{Ker}(\psi_{|_C}),$$}}
	see \cite{green-koszul}, we have $\alpha' \in \mathrm{Ker}(\psi)$ as required.\smallskip
	
	From Lemma \ref{E2} we have $\phi_{|_{\mathrm{H}^0(A) \otimes \mathrm{H}^0(A)}} \seq \C\langle t^2 \rangle \seq \mathrm{H}^0(L^{\otimes 2})$, where $t \in \mathrm{H}^0(L)$ defines $C$ and 
	{\small{\begin{align*}
	\phi\left((x\otimes t) \wedge(y \otimes q(xw)-x\otimes q(yw))\right)&=tq(xw)\\
		\phi\left((y\otimes t) \wedge(y \otimes q(xw)-x\otimes q(yw))\right)&=tq(yw).
\end{align*}}}
It follows that {\small{$$j\circ \mathrm{det}_X(\bigwedge^2 \mathrm{H}^0(E)) \seq \C \langle t \rangle \seq \mathrm{H}^0(L^{\otimes 2}).$$}}
Thus, up to a nonzero scalar, $j$ is multiplication by $t \in \mathrm{H}^0(L)$ and we have the formulae {\small{$$\mathrm{det}_X=\frac{1}{t} \phi, \; \; \mathrm{det}={\mathrm{det}_X}_{|_C}.$$}} Since $t$ vanishes on $C$ the above calculations for the map $\phi$ give that $\mathrm{det}_{|_{\mathrm{H}^0(A) \otimes \mathrm{H}^0(A)}}=0$ 
and $\mathrm{det}_{|_{\mathrm{H}^0(A)\otimes \mathrm{H}^0(\omega_C \otimes A^{\vee})}}$ is the Petri map.
	\end{proof}
We now use the determinant map to define a morphism from the Grassmannian of planes in $\mathrm{H}^0(E))$ to the linear system $|L|$ as in \cite{V1}.
\begin{prop} \label{d}
$E$ be the rank two Lazarsfeld--Mukai bundle from Lemma \ref{E}. We have a natural, finite morphism 
$$d \; : \; \mathrm{Gr}_2(\mathrm{H}^0(E)) \to |L|,$$
which is surjective, flat and has degree $\frac{(2k)!}{k!(k+1)!}$. Further, the fibre of $d$ over a curve $C \in |L|$ may be identified naturally with the Brill--Noether space $W^1_{k+1}(C)$. 
\end{prop}
\begin{proof}
We follow \cite{V1}. The determinant map $\mathrm{det} \; : \; \wedge^2 \mathrm{H}^0(X,E) \to \mathrm{H}^0(X,L)$ does not vanish on any element of the form $s \wedge t$ for $s, t \in  \mathrm{H}^0(X,E)$, see \cite{V1}, proof of equation $(3.18)$. We therefore have a well-defined morphism
{\small{\begin{align*}
d \; : \; \mathrm{Gr}_2(\mathrm{H}^0(E)) &\to |L| \\
W &\mapsto [\mathrm{det}(\wedge^2 W)].
\end{align*}}}
The pull-back of the hyperplane class of $|L|$ by $d$ is the Pl\"ucker hyperplane class of $\mathrm{Gr}_2(\mathrm{H}^0(E))$, \cite[Lemma 2.1]{ogrady}. Since the Pl\"ucker hyperplane class is very ample, $d$ is finite. Since the Grassmannian $\mathrm{Gr}(2,n)$ has degree given by the Catalan number $\frac{(2(n-2))!}{(n-2)!(n-1)!}$, $d$ has degree $\frac{(2k)!}{k!(k+1)!}$. As $\dim  \mathrm{Gr}_2(\mathrm{H}^0(E))=\dim |L|=2g$, the morphism $d$ is surjective and flat by \cite[Ex.\ III.9.3]{hartshorne}.\smallskip

Let $\pi : \mathcal{W}^1_{k+1} \to |L|$ be the relative Brill--Noether variety, which is constructed as in \cite[CH.\ XXI]{ACG2}. For any $C \in |L|$, $\pi^{-1}(C)=W^1_{k+1}(C)$. We have a morphism $\phi: \mathcal{W}^1_{k+1}  \to \mathrm{Gr}_2(\mathrm{H}^0(E))$ defined as such. Given a torsion-free sheaf $A \in W^1_{k+1}(C)$, for$C \in |L|$, we have an exact sequence
{\small{$$0 \to E \to \mathrm{H}^0(A) \otimes L \to A \otimes \omega_C \to 0.$$}}
If $t \in \mathrm{H}^0(L)$ defines $C$, then the image of $\mathrm{H}^0(A) \otimes \mathcal{O}_X \xrightarrow{\otimes t} \mathrm{H}^0(A) \otimes L$ lands in $E$, and therefore induces an inclusion $i_A: \mathrm{H}^0(A) \otimes \mathcal{O}_X  \hookrightarrow E$ by Lemma \ref{E2}. We set $\phi([A])=[\mathrm{H}^0(i_A)]$. Notice also that $\mathrm{Cok}(i_A) \simeq A^{-1} \otimes \omega_C$.\smallskip

If $C \in |L|$ is general, then both $\pi^{-1}(C)$ and $d^{-1}(C)$ are zero-dimensional and nonempty. Hence $\phi$ is surjective. Thus, if $[V \hookrightarrow \mathrm{H}^0(E)] \in \mathrm{Gr}_2(\mathrm{H}^0(E))$, the cokernel of the natural map $V \otimes \mathcal{O}_X \to E$ is of the form $A^{-1}_V \otimes \omega_C$ for some $A_V \in W^1_{k+1}(C')$ for some $C' \in |L|$. Thus we define an inverse $\psi : \mathrm{Gr}_2(\mathrm{H}^0(E)) \to \mathcal{W}^1_{k+1} $ to $\phi$  by $\psi([V]):= [A_V]$. Thus $\phi$ is an isomorphism $\mathcal{W}^1_{k+1}  \simeq \mathrm{Gr}_2(\mathrm{H}^0(E))$. This completes the proof.

\end{proof}
\begin{lem} \label{res-1d-kos}
With notation as in Proposition \ref{d}, let $A \in W^1_{k+1}(C) \simeq d^{-1}(C)$ for $C \in |L|$ and let $s \neq 0 \in \mathrm{H}^0(C,A)$. The restriction map $\mathrm{res} : \mathrm{H}^0(X,L) \twoheadrightarrow \mathrm{H}^0(C,\omega_C)$ induces an isomorphism
$$\mathrm{K}_{k-1,1}(C,\omega_C;\mathrm{H}^0(\omega_C\otimes A^{-1})) \simeq \mathrm{K}_{k-1,1}(X,L;\mathrm{H}^0(L \otimes I_{Z(s)})).$$
\end{lem}
\begin{proof}
By the Lefschetz Theorem \cite{green-koszul}, $\mathrm{res}$ induces an isomorphism $\alpha: \mathrm{K}_{k-1,1}(X,L) \xrightarrow{\sim} \mathrm{K}_{k-1,1}(C,\omega_C)$.
We have the exact sequence $0 \to \mathrm{H}^0(A) \otimes \mathcal{O}_X \to E \to A^{-1} \otimes \omega_C \to 0$ by the proof of Lemma \ref{E2}. 
Now consider the commutative diagram
{\small{$$\begin{tikzcd}
	 	 & & & 0 \arrow[d] & \\
 & &&  \mathcal{O}_X \arrow[d, "\otimes t"]  &\\
	0 \arrow[r] & \mathcal{O}_X \arrow[r] \arrow[d, "\otimes s"]& E \arrow[r] \arrow[d, "\simeq"] & L \otimes I_{Z(s)} \arrow[r] \arrow[d] & 0\\
	0 \arrow[r] &  \mathrm{H}^0(A) \otimes \mathcal{O}_X  \arrow[r] &  E \arrow[r] &A^{-1}\otimes \omega_C \arrow[r] \arrow[d]& 0\\
		 	 &  & &0&
\end{tikzcd}$$}}
where $t \in \mathrm{H}^0(L \otimes I_{Z_(s)})$ defines $C$. It follows that $\mathrm{res}^{-1}(\mathrm{H}^0(\omega_C\otimes A^{-1})=\mathrm{H}^0(L \otimes I_{Z(s)})$ and $\alpha$ maps the subspace $\mathrm{K}_{k-1,1}(X,L;\mathrm{H}^0(L \otimes I_{Z(s)})) \seq \mathrm{K}_{k-1,1}(X,L) $ isomorphically to $\mathrm{K}_{k-1,1}(C,\omega_C;\mathrm{H}^0(\omega_C\otimes A^{-1})) \seq \mathrm{K}_{k-1,1}(C,\omega_C)$.
\end{proof}

In the following proposition, we verify the assumptions of Proposition \ref{definability-map-brill-syz} in a special case.
\begin{prop} \label{assumps-OK}
	Let $(X,L)$ be a very general, primitively polarized, K3 surface of even genus $g=2k$ for $g \geq 2$. Let $C \in |L|$ be an integral curve. For any $[A] \in W^1_{k+1}(C)$ and $s \neq 0 \in \mathrm{H}^0(C,A)$, the map 
	{\small{$$ \delta_s \; : \; \bigwedge^{k}\mathrm{H}^0(\omega_C \otimes A^{\vee}) \otimes \frac{ \mathrm{H}^0(A)}{\C \langle s\rangle } \to \mathrm{K}_{k-1,1}(C,\omega_C; \mathrm{H}^0(\omega_C \otimes A^{\vee}))$$}} from Lemma \ref{natural-injection} is injective.
	\end{prop}
\begin{proof}
	 Let $E$ be the rank-two Lazarsfeld--Mukai bundle from Lemma \ref{E}. The dual $E^{\vee}$ fits into the exact sequence 
$0 \to {E}^{\vee} \to \mathrm{H}^0(A) \otimes \mathcal{O}_X \to i_*A \to 0.$ Let $Z=Z(s) \seq C$ denote the zero locus of $s$. From the surjective restriction map $\mathrm{H}^0(X,L) \to \mathrm{H}^0(C,\omega_C)$, we have isomorphisms 
		{\small{\begin{align*}
		\mathrm{Coker}\left(\mathrm{H}^0(X,L) \to \mathrm{H}^0(Z,L_{|_Z}) \right)&\simeq \mathrm{Coker}\left(\mathrm{H}^0(C,\omega_C) \to \mathrm{H}^0(Z,L_{|_Z}) \right)\\
		& \simeq \mathrm{Ker}\left(\mathrm{H}^1(C,\omega_C\otimes A^{\vee}) \to \mathrm{H}^1(C,\omega_C) \right),
		\end{align*}}}
		where the given maps are the natural ones, using that $A^{\vee} \simeq I_{Z/C}$. The map $\mathrm{H}^1(C,\omega_C\otimes A^{\vee}) \to \mathrm{H}^1(C,\omega_C)$ is surjective and has one-dimensional kernel by Serre duality. Thus, $ \mathrm{H}^0(X,L) \to \mathrm{H}^0(Z,L_{|_Z}) $ has one dimensional cokernel. Further, $W^1_d(C)=\emptyset$ for all $d<k+1$ by Lemma \ref{E}. Applying the above argument to any proper subscheme $Z' \subsetneq Z$, and using that $\mathrm{H}^0(I_{Z'/C}^{\vee})<2$ (as $W^1_{d(Z'))}(C)=\emptyset$) shows that the natural restriction map 
		{\small{$$\mathrm{H}^0(X,L) \to \mathrm{H}^0(Z',L_{|_{Z'}})$$}}
		is surjective. By Serre's construction, \cite[\S 3]{lazarsfeld-lectures}, we can construct a rank two vector bundle on $X$ associated to $Z \seq X$. This bundle has the same invariants as $E$ and is stable, which implies that it must be isomorphic to $E$, see \cite[\S 2]{V1}. Thus we have an exact sequence
		{\small{$$0 \to \mathcal{O}_X \xrightarrow{s'} E \to L \otimes I_{Z/X} \to 0,$$}}
		where the section $s'\in \mathrm{H}^0(E)$ is uniquely determined up to a nonzero scalar multiple. Since $\mathrm{H}^1(\mathcal{O}_X)=0$, we obtain an exact sequence $0 \to \mathrm{H}^0(\mathcal{O}_X )\xrightarrow{s'} \mathrm{H}^0(E) \xrightarrow{\pi} \mathrm{H}^0(L \otimes I_{Z/X}) \to 0.$
		\smallskip
		
		Following \cite[\S 3.4.1]{aprodu-nagel}, we have a natural, nonzero map 
		{\small{$$\gamma \; : \; \bigwedge^{k+1} \mathrm{H}^0(X, L \otimes I_{Z/X}) \to \mathrm{K}_{k-1,1}(X,L; \mathrm{H}^0(X, L \otimes I_{Z/X}))$$}}
		defined as such: let $U \seq \mathrm{H}^0(E) $ be any subspace with $\pi_{|_U}: U \to  \mathrm{H}^0(L \otimes I_{Z/X})$ an isomorphism. Then, since $\rank(U)=k+1$, we have natural isomorphisms
		{\small{\begin{align*}
		\mathrm{Hom}_k (\wedge^{k+1} \mathrm{H}^0( L \otimes I_{Z/X}), \wedge^{k-1} \mathrm{H}^0( L \otimes I_{Z/X}) \otimes \mathrm{H}^0(L))  &\simeq \wedge^{k-1}U \otimes \wedge^{k+1}U^{\vee} \otimes \mathrm{H}^0(L)\\
		&\simeq\wedge^2 U^{\vee} \otimes \mathrm{H}^0(L)\\
		&\simeq \mathrm{Hom}_k (\wedge^2 U, \mathrm{H}^0(L)).
		\end{align*}}}
		Let $\hat{\gamma} \in \mathrm{Hom}_k (\wedge^{k+1} \mathrm{H}^0( L \otimes I_{Z/X}), \wedge^{k-1} \mathrm{H}^0( L \otimes I_{Z/X}) \otimes \mathrm{H}^0(L))$ denote the element corresponding to $\mathrm{det}_{U} \in \mathrm{Hom}_k (\wedge^2 U, \mathrm{H}^0(L))$, where $\mathrm{det}$ is the natural map $\mathrm{det} : \wedge^2 \mathrm{H}^0(E) \to \mathrm{H}^0(\wedge^2 E) \simeq \mathrm{H}^0(L).$ It is shown in \cite[\S 3.4.1]{aprodu-nagel} that $\hat{\gamma}$ induces a nonzero map $\gamma : \wedge^{k+1} \mathrm{H}^0(L \otimes I_{Z/X}) \to \mathrm{K}_{k-1,1}(X,L; \mathrm{H}^0(L \otimes I_{Z/X}))$, and that this map does not depend on the choice of $U$.\smallskip
		
		Since $\gamma$ is nonzero and $h^0(L \otimes I_{Z/X}))=k+1$, the map $\gamma$ is injective. To complete the proof, it remains to show that we may identify $\gamma$ with $\delta_s$. From the exact sequence
		{\small{$$0 \to I_{C/X} \to I_{Z/X} \to I_{Z/C} \to 0,$$}}
		we obtain, after twisting by $L=I_{C/X}^{\vee}$, a short exact sequence
		{\small{$$0 \to \mathcal{O}_X \xrightarrow{t} L \otimes I_{Z/X} \to \omega_C \otimes A^{\vee} \to 0,$$}}
		where $t \in \mathrm{H}^0( L \otimes I_{Z/X})$ is a section defining the curve $C$. Taking global sections we obtain a (non-canonical) isomorphism
		$\wedge^{k+1} \mathrm{H}^0(L \otimes I_{Z/X}) \simeq \wedge^{k} \mathrm{H}^0(\omega_C \otimes A^{\vee})\otimes {\C \langle t \rangle } .$\smallskip
		
		We will now show that there is a natural isomorphism ${\C \langle t \rangle } \simeq \frac{ \mathrm{H}^0(A)}{\C \langle s\rangle }$. By Remark \ref{rem-EC}, we have a commutative diagram
		{\small{$$\begin{tikzcd}
				&&& 0\arrow[d]&\\
&0\arrow[d]&& \mathrm{H}^0(\mathcal{O}_X )\arrow[d,"t"]&\\
0 \arrow[r] & \mathrm{H}^0(\mathcal{O}_C)\simeq \mathrm{H}^0(\mathcal{O}_X )\arrow[r, "s' "] \arrow[d, "s"] & \mathrm{H}^0(E) \arrow[r, "\pi"] \arrow[d, "\simeq"] &\mathrm{H}^0(L \otimes I_{Z/X}) \arrow[r] \arrow[d, "q"] & 0\\
0 \arrow[r] & \mathrm{H}^0(A) \arrow[r] \arrow[d]  & \mathrm{H}^0(E_C) \arrow[r, "\widetilde{\pi}"] &\mathrm{H}^0(\omega_C \otimes A^{-1}) \arrow[r] \arrow[d] & 0\\
& \frac{ \mathrm{H}^0(A)}{\C \langle s\rangle }\arrow[d]&&0&\\
& 0 &&&
\end{tikzcd}$$}}
with exact rows and columns, which gives the claimed isomorphism ${\C \langle t \rangle } \simeq \frac{ \mathrm{H}^0(A)}{\C \langle s\rangle }$. Now let $\{s', t', v'_1, \ldots, v'_k\}$ be a basis for $\mathrm{H}^0(E)$ with $\pi(t')=t$, and $v_i:=\pi(v_i)$, $1 \leq i \leq k$. By \cite[\S 3.4.1]{aprodu-nagel}, we have the formula (up to sign)
{\small{\begin{align*}
\hat{\gamma}(v_1 \wedge \ldots \wedge v_k \wedge t)&=\sum_{i<j} (-1)^{i+j} v_1 \wedge \ldots \wedge \hat{v_i}\wedge \ldots \wedge \hat{v_j} \ldots \wedge v_k \wedge t \otimes \mathrm{det}(v'_i\wedge v_j)\\
&+\sum_{i\leq k} (-1)^{i+k+1} v_1 \wedge \ldots \wedge \hat{v_i}\wedge \ldots \wedge v_k \otimes \mathrm{det}(v'_i\wedge t).
\end{align*}}}
We have a natural isomorphism $ \mathrm{K}_{k-1,1}(X,L; \mathrm{H}^0( L \otimes I_{Z/X})) \simeq  \mathrm{K}_{k-1,1}(C,\omega_C; \mathrm{H}^0(C \otimes A^{-1}))$ induced by restriction to $C$, \cite[\S 2]{kemeny-voisin}. Let $\widetilde{\gamma}$ be the composition of $\gamma$ with this identification. Since $t$ lies in the kernel of the restriction map $q: \mathrm{H}^0( L \otimes I_{Z/X}) \to \mathrm{H}^0(C \otimes A^{-1}))$, $\widetilde{\gamma}$ takes $v_1 \wedge \ldots \wedge v_k \wedge t$ to the equivalence class of 
{\small{\begin{align*}
&\sum_{i\leq k} (-1)^{i+k+1} q(v_1) \wedge \ldots \wedge \widehat{q(v_i)}\wedge \ldots \wedge q(v_k) \otimes \mathrm{det}_C(v'_i\wedge t')\\
&=\sum_{i\leq k} (-1)^{i+k+1} q(v_1) \wedge \ldots \wedge \widehat{q(v_i)}\wedge \ldots \wedge q(v_k) \otimes \widetilde{\pi}(v'_i)\otimes t',
\end{align*}}}
where we treat $t' \in \mathrm{Ker}(H^0(E_C) \to \mathrm{H}^0(\omega_C \otimes A^{-1}))$ as an element of $\mathrm{H}^0(A)$ and use that $\mathrm{det}_C$ coincides with the natural map $\mathrm{H}^0(A) \otimes \mathrm{H}^0(\omega_C\otimes A^{-1}) \to \mathrm{H}^0(\omega_C)$
on the component $\mathrm{H}^0(A) \otimes \mathrm{H}^0(\omega_C\otimes A^{-1})$, and is zero on all other components, by Lemma \ref{det-petri}.
Comparing with Lemma \ref{natural-injection}, we see that we may identify $\gamma$ with $\delta_s$, as required.
\end{proof}

By Lemma \ref{assumps-OK} and Proposition \ref{definability-map-brill-syz}, if $(X,L)$ is a very general, primitively polarized, K3 surface of even genus $g=2k$ for $g \geq 2$ and $C \in |L|$ be an integral curve, then we have a well-defined morphism $S :  X^1_{k+1}(C)  \to \PP(\mathrm{K}_{k-1,1}(C,\omega_C))$.
\begin{prop} \label{nondeg-K3curve}
Set $\PP:=\PP(\mathrm{K}_{k-1,1}(C,\omega_C))$. With the above notation, the restriction map
$$S^* \; : \; \mathrm{H}^0(\mathcal{O}_{\PP}(1)) \to \mathrm{H}^0(S^*\mathcal{O}_{\PP}(1))$$
is injective.
\end{prop}
\begin{proof}
Let $G:=\mathrm{Gr}_2(\mathrm{H}^0(E))$, where $E$ is the vector bundle from Lemma \ref{E}. We have the tautological subbundle $F \hookrightarrow \mathrm{H}^0(E) \otimes \mathcal{O}_G$ on the Grassmannian $G$, with $F \otimes k(p)\simeq V \seq \mathrm{H}^0(E)$ for $p=[V \seq \mathrm{H}^0(E)] \in G$. Let $$p \; : \; \PP(F) \to |L|$$
be the composite of the projection $\PP(F) \to G$ with the morphism $d : G \to |L|$ from Proposition \ref{d}. We may naturally identify $p^{-1}(C)$ with $X^1_{k+1}(C)$. We have the natural map, $q: \PP(F) \to \PP(\mathrm{H}^0(E))$ as well as a morphism $\psi  : \PP(\mathrm{H}^0(E)) \to \PP(\mathrm{K}_{k-1,1}(X,L))$
defined by $$\psi([s]):= [\mathrm{K}_{k-1,1}(X,L,\mathrm{H}^0(L \otimes I_{Z(s)}))],$$ \cite[Thm.\ 2]{kemeny-voisin}. By Lemma \ref{res-1d-kos}, for $s \in \mathrm{H}^0(A) \in W^1_{k+1}(C)$, we have an isomorphism
$$\mathrm{K}_{k-1,1}(C,\omega_C;\mathrm{H}^0(\omega_C\otimes A^{-1})) \simeq \mathrm{K}_{k-1,1}(X,L;\mathrm{H}^0(L \otimes I_{Z(s)})),$$
where we view $\mathrm{H}^0(A) $ as a subspace of $\mathrm{H}^0(E)$ via Proposition \ref{d}. In particular, $$\dim_{\C}\mathrm{K}_{k-1,1}(C,\omega_C;\mathrm{H}^0(\omega_C\otimes A^{-1})) =\dim_{\C} \mathrm{K}_{k-1,1}(X,L;\mathrm{H}^0(L \otimes I_{Z(s)}))=1,$$
so the injective map $\delta_s$ from Proposition \ref{assumps-OK} is an isomorphism. Furthermore, $S$ is identified with 
$\psi \circ q_{|_{p^{-1}(C)}}.$ By \cite[Thm.\ 2]{kemeny-voisin}, the map
$\psi^* : \mathrm{H}^0(\mathcal{O}_{\PP(\mathrm{K}_{k-1,1}(X,L))}(1)) \to \mathrm{H}^0(\mathcal{O}_{ \PP(\mathrm{H}^0(E)) }(k-2)),$
is an isomorphism. We further have the isomorphism
$$q^* \; : \; \mathrm{H}^0(\mathcal{O}_{ \PP(\mathrm{H}^0(E)) }(k-2)) \simeq \mathrm{Sym}^{k-2}\mathrm{H}^0(E)^{\vee} \to \mathrm{H}^0(\mathcal{O}_{ \PP(F)}(k-2))\simeq \mathrm{H}^0(G,\mathrm{Sym}^{k-2}F^{\vee}).$$
We further have a natural identification $\mathrm{H}^0(p^{-1}(C),\mathcal{O}_{ \PP(F)}(k-2))\simeq \mathrm{H}^0(d^{-1}(C), \mathrm{Sym}^{k-2}F^{\vee})$.
So it suffices to show that the restriction map 
$$\mathrm{H}^0(G,\mathrm{Sym}^{k-2}F^{\vee}) \to \mathrm{H}^0(d^{-1}(C), \mathrm{Sym}^{k-2}F^{\vee})$$
is injective. This is proven in \cite{V1}, proof of Proposition 7, by a Koszul complex computation, noting that $d^{-1}(C)$ is a complete intersection of hyperplane sections of $G$.
\end{proof}

We now deform to a \emph{general} nodal curve, not necessarily lying on a K3 surface. Set $g=2k\geq 2$, fix $0 \leq n \leq 2k$, and let $C$ be a general, integral, $n$-nodal curve, i.e.\ $C$ has precisely $n$ nodes and no other singularities. By Proposition \ref{gen-node-BN}, $C$ has gonality $k+1$. Further, $W^1_{k+1}(C)$ is zero-dimensional and reduced, and any closed point $[A] \in W^1_{k+1}(C)$ corresponds to a \emph{locally free sheaf} $A$ on $C$. We additionally have $W^2_{k+1}(C)=\emptyset$. By Proposition \ref{one-dim-iso}, the natural map
	{\small{$$ \delta_s \; : \; \wedge^{k}\mathrm{H}^0(\omega_C \otimes M^{\vee}) \otimes \frac{ \mathrm{H}^0(M)}{\C \langle s\rangle } \to \mathrm{K}_{k-1,1}(C,\omega_C; \mathrm{H}^0(\omega_C \otimes M^{\vee}))$$ }}
	for any line bundle $[M] \in W^1_{k+1}(C)$ and $s \neq 0 \in \mathrm{H}^0(C,M)$ is an isomorphism. By Proposition \ref{definability-map-brill-syz}, we have a well-defined morphism $S : X^1_{k+1}(C) \to \PP(\mathrm{K}_{k-1,1}(C,\omega_C)).$ We first need a lemma.
\begin{lem} \label{rel-koszul}
Let $(B,0)$ be a smooth, integral, pointed variety of dimension one. Let $\pi: \mathcal{C} \to B$ be a flat, projective family of integral, nodal curves, with $\mathcal{C}_b:=\pi^{-1}(b)$ for $b \in B$. Set $C:=\mathcal{C}_0$. Suppose $\omega_{C}$ is globally generated and $\mathrm{K}_{k,1}(C,\omega_{C})=0$. Then, after replacing $(B,0)$ by an open set about $0$, we have a vector bundle $\mathcal{K}$ on $B$, with natural isomorphisms $\mathcal{K}_b\simeq \mathrm{K}_{k-1,1}(\mathcal{C}_b,\omega_{\mathcal{C}_b})$ for all $b \in B$. 
\end{lem}
\begin{proof}
By Proposition \ref{semi-cont} we may assume that $\mathcal{C}_b$ is globally generated and that $\mathrm{K}_{k,1}(\mathcal{C}_b,\omega_{\mathcal{C}_b})=0$ for all $b \in B$. In particular, this implies that $(\mathcal{C}_b, \omega_{\mathcal{C}_b})$ has the Betti table of a general canonical curve, for all $b \in B$. In particular, $\beta_{p,1}:=\dim \mathrm{K}_{p,1}(\mathcal{C}_b,\omega_{\mathcal{C}_b})$ is constant and independent of $b \in B$ for any $p$. Let $M_b$ denote the kernel bundle
$0 \to M_b \to \mathrm{H}^0(\omega_{\mathcal{C}_b}) \otimes \mathcal{O}_{\mathcal{C}_b} \to \omega_{\mathcal{C}_b} \to 0.$
We have a natural isomorphism
{\small{\begin{align*}
\text{Ker}\left(d: \wedge^{k-1} \mathrm{H}^0(\omega_{\mathcal{C}_b}) \otimes \mathrm{H}^0(\omega_{\mathcal{C}_b}) \to \wedge^{k-2} \mathrm{H}^0(\omega_{\mathcal{C}_b}) \otimes \mathrm{H}^0(\omega_{\mathcal{C}_b}^{\otimes 2}) \right) \simeq \mathrm{H}^0(\wedge^{k-1}M_b \otimes \omega_{\mathcal{C}_b}),
\end{align*}}}
see \cite[\S 2.3]{aprodu-farkas-clay}. Further, we may naturally identify $\mathrm{K}_{k-1,1}(\mathcal{C}_b,\omega_{\mathcal{C}_b})$ with
{\small{$$\mathrm{Coker}\left(\wedge^k \mathrm{H}^0(\omega_{\mathcal{C}_b}) \to \mathrm{H}^0(\wedge^{k-1}M_b \otimes \omega_{\mathcal{C}_b})\right),$$}}
where the map $\wedge^k \mathrm{H}^0(\omega_{\mathcal{C}_b}) \to \mathrm{H}^0(\wedge^{k-1}M_b \otimes \omega_{\mathcal{C}_b})$ is injective. Since $h^0(\omega_{\mathcal{C}_b})$ and $\beta_{p,1}$ are independent of $b \in B$, $h^0(\wedge^{k-1}M_b \otimes \omega_{\mathcal{C}_b})$ is constant for $b \in B$. We have a vector bundle $\mathcal{M}$ on $\mathcal{C}$ fitting into the exact sequence $0 \to \mathcal{M} \to \pi^* \pi_* \omega_{\pi} \to \omega_{\pi} \to 0,$
with fibre $\mathcal{M}_b \simeq M_b$ on $\mathcal{C}_b$. Consider now the vector bundles $\mathcal{A}:=\wedge^k \pi_* \omega_{\pi}$, $\mathcal{B}:=\pi_*(\wedge^{k-1}\mathcal{M} \otimes \omega_{\pi})$. We have an injective morphism $\mathcal{A} \to \mathcal{B}$ relativizing $\wedge^k \mathrm{H}^0(\omega_{\mathcal{C}_b}) \to \mathrm{H}^0(\wedge^{k-1}M_b \otimes \omega_{\mathcal{C}_b})$, for $b \in B$. By Lemma \ref{coker-bundle}, we have a vector bundle $\mathcal{K}:=\mathrm{Coker}(\mathcal{A} \to \mathcal{B})$. Then $\mathcal{K}_b\simeq \mathrm{K}_{k-1,1}(\mathcal{C}_b,\omega_{\mathcal{C}_b})$ for all $b \in B$, as required.
\end{proof}	
	
We now generalize Proposition \ref{nondeg-K3curve} to the situation of a general nodal curve. 
\begin{prop} \label{gen-nondegK3}
Let $C$ be a general, integral, $m$-nodal curve of even arithmetic genus $g=2k \geq 2$ for $0 \leq m \leq k-1$.  The pull-back map on global sections
$$S^* \; : \; \mathrm{H}^0(\mathcal{O}_{\PP}(1)) \to \mathrm{H}^0(S^*\mathcal{O}_{\PP}(1)),$$
with $\PP:=\PP(\mathrm{K}_{k-1,1}(C,\omega_C))$, is injective.
\end{prop}
\begin{proof}
A very general, primitively polarized K3 surface $(X,L)$ of genus $g=2k$ contains $n$-nodal curves in the linear system $|L|$ for any $n$ by Chen's Theorem, \cite{chen-every-rational} (see also \cite[Theorem 3.8]{tannenbaum}). So let $(X,L)$ be a very general, primitively polarized K3 surface and let $D \in |L|$ be an integral, $n$-nodal curve. Let $\pi: \mathcal{C} \to B$ be flat family of integral, $n$-nodal curves of genus $2k$ over a smooth, integral, pointed variety $(B,0)$ of dimension one, with $D\simeq \pi^{-1}(0)$ the special fibre and $\mathcal{C}_b:=\pi^{-1}(b)$ a general, integral $n$-nodal curve for $b \neq 0 \in B$. We may also assume that $\pi$ admits a section $B \to \mathcal{C}$, landing in the smooth locus of each fibre. We have a moduli space $\mathcal{W}^1_{k+1} $ and a projective morphism $q: \mathcal{W}^1_{k+1} \to B$ such that $q^{-1}(b)\simeq W^1_{k+1}(\mathcal{C}_b)$ by relativizing the construction of the Brill--Noether loci as in \cite[Ch.\ XXI]{ACG2}.\smallskip

We first show that the morphism $q: \mathcal{W}^1_{k+1} \to B$ is flat. The claim will then follow easily from Proposition \ref{nondeg-K3curve}. Note that $q$ is finite by Proposition \ref{gen-node-BN} and Proposition \ref{d}. For $b \neq 0$, $W^1_{k+1}(\mathcal{C}_b)$ is smooth of dimension zero by Proposition \ref{gen-node-BN}. After replacing the curve $B$ by an open set about the origin $0$, we may assume that $q^{-1}(U)\to U$ is \'etale, for $U=B\setminus \{0\}$, and that $h^0(\mathcal{O}_{W^1_{k+1}(\mathcal{C}_b)})=c$ is constant for $b \neq 0 \in B$. It thus suffices to show $h^0(\mathcal{O}_{W^1_{k+1}(D)})=c$, by \cite[Thm.\ III.9.9]{hartshorne}. Since the closure of $q^{-1}(U)$ in $\mathcal{W}^1_{k+1}$ is flat over $B$ by \cite[Prop.\ III.9.7]{hartshorne}, we have $h^0(\mathcal{O}_{W^1_{k+1}(D)})\geq c$. By Proposition \ref{d}, $h^0(\mathcal{O}_{W^1_{k+1}(D)})=\frac{(2k)!}{k!(k+1)!}$. So it suffices to show that $c \geq \frac{(2k)!}{k!(k+1)!}$. But, if $Y$ is a general, smooth curve of genus $2k$, then $W^1_{k+1}(Y)$ is reduced and zero-dimensional of length $\frac{(2k)!}{k!(k+1)!}$, \cite{griffiths-harris}. It follows that $c \geq \frac{(2k)!}{k!(k+1)!}$, by degenerating a general smooth curve to the $n$-nodal curve $\mathcal{C}_b$ (for some $b \neq 0$) and applying \cite[Prop.\ III.9.7]{hartshorne} again.\smallskip

So $q: \mathcal{W}^1_{k+1} \to B$  is finite and flat. By Proposition \ref{gen-node-BN} and Lemma \ref{E}, for any $b \in B$ and any $[A] \in W^1_{k+1}(\mathcal{C}_b)$, $h^0(A)=2$. We thus have a $\PP^1$ bundle $\mathcal{X}^1_{k+1}$ over $\mathcal{W}^1_{k+2}$, with a morphism $r: \mathcal{X}^1_{k+1} \to B$, with fibre $r^{-1}(b) \simeq X^1_{k+1}(\mathcal{C}_b)$. By Lemma \ref{rel-koszul} we have, after shrinking $B$, a vector bundle $\mathcal{K}$ with natural isomorphisms $\mathcal{K}_b\simeq \mathrm{K}_{k-1,1}(\mathcal{C}_b,\omega_{\mathcal{C}_b})$ for all $b \in B$. Proposition \ref{definability-map-brill-syz} relativizes to give a morphism $$\mathcal{S}: \mathcal{X}^1_{k+1} \to \PP(\mathcal{K})$$ over $B$. By Proposition \ref{nondeg-K3curve}, $\mathcal{S}_b^*$ is injective for all $b \in B$, after shrinking $B$, which gives the claim.
\end{proof}

\begin{prop} \label{last-ruling-prop}
Let $D$ be a integral $n$-nodal curve of even arithmetic genus $g=2k$ with $\omega_D$ very ample and $(D,\omega_D)$ normally generated. Let $A \in W^1_{k+1}(D)$ be a base-point free line bundle with $h^0(A)=2$. Let $\alpha \neq 0 \in K_{k-1,1}(D,\omega_D; \mathrm{H}^0(\omega_D \otimes I_{Z(s)}))$ for $s \neq 0 \in \mathrm{H}^0(A)$. Then the associated syzygy scheme $X_{\alpha}$ is a scroll and $D$ lies in the smooth locus of $X_{\alpha}$, with associated line bundle
$$A \simeq L_{\alpha}:=\mathcal{O}_D(R),$$
where $R$ is the ruling on $X_{\alpha}$.
\end{prop}
\begin{proof}
Let $X_A \seq \PP^{g-1}$ be the scroll induced by $A$, \cite[\S 2]{schreyer1}. Then $X_A$ is a rational normal scroll of degree $h^1(A)=k$ and dimension $k$ containing $D \seq \PP^{g-1}$ and further $D$ lies in the smooth locus of $X_{\alpha}$ since $A$ is base-point free, \cite[\S 4]{lin-syz}. We have the injective restriction map
$$i \; : \mathrm{K}_{k-1,1}(X_A,H) \hookrightarrow \mathrm{K}_{k-1,1}(D,\omega_D)$$
where $H$ is the hyperplane. Notice that $(\mathcal{O}_{X_A}(R))_{|_D} \simeq A$ and that the $k$ dimensional vector space $\mathrm{H}^0(X_A, H-R)$ is map isomorphically to $\mathrm{H}^0(D,\omega_D-A)$ under the natural identification $\mathrm{H}^0(X_A, H) \simeq \mathrm{H}^0(D,\omega_D)$. Let $s'\in \mathrm{H}^0(X,R)$ be the unique section with $s'_{|_D}=s$. The space $\mathrm{K}_{k-1,1}(X_A,H; \mathrm{H}^0(H \otimes I_{Z(s')})) $ is one dimensional from Section \ref{rank-section} and injects into the one dimensional space $\mathrm{K}_{k-1,1}(D,\omega_D; \mathrm{H}^0(\omega_D \otimes I_{Z(s)}))$. Hence $i$ induces an isomorphism $$\mathrm{K}_{k-1,1}(X_A,H; \mathrm{H}^0(H \otimes I_{Z(s')})) \simeq \mathrm{K}_{k-1,1}(D,\omega_D; \mathrm{H}^0(\omega_D \otimes I_{Z(s)})).$$ In particular, for $\alpha \neq 0 \in  \mathrm{K}_{k-1,1}(D,\omega_D; \mathrm{H}^0(\omega_D \otimes I_{Z(s)}))$ there exists $\alpha' \in \mathrm{K}_{k-1,1}(X_A,H)$ with $i(\alpha')=\alpha$. By definition of the syzygy scheme $X_{\alpha}$ we therefore have $X_{A} \seq X_{\alpha}$. But $X_{\alpha} \seq \PP^{g-1}$ is also a rational normal scroll of degree $k$, \cite[Cor.\ 5.2]{bothmer-JPAA} and hence $X_{\alpha}=X_A$. This completes the proof.
\end{proof}

We may now prove the main result of this paper.
\begin{thm}
Let $C$ be a general curve of genus $g \geq 8$. Then, for any $p$, $\mathrm{K}_{p,1}(C,\omega_C)$ is spanned by syzygies of rank $p+1$.
\end{thm}
\begin{proof}
By Voisin's Theorem \cite{V2} (see also \cite{kemeny-voisin} for a simpler proof), $\mathrm{K}_{p,1}(C,\omega_C)=0$ for $p \geq \lfloor \frac{g}{2} \rfloor$. Thus the result is trivial unless $p \leq \lfloor \frac{g-2}{2} \rfloor$. Fix such a $p$ and set $m=g-2p-2$. Let $x_i, y_i \in C$ be general points for $1 \leq i \leq m$. \smallskip

For any $1 \leq j \leq m$, we let $D_j$ be the $j$-nodal curve obtained by identifying $x_i$ to $y_i$ for $1 \leq i \leq j$. We let $p_i \in D_j$ denote the node over $x_i, y_i$ for $1 \leq i \leq j$. Set $D_0:=C$. We let
$$\mu_j \; : \; D_{j-1} \to D_j$$
be the partial normalization map at the node $p_j$. By Proposition \ref{induction-step} it suffices to show that $\mathrm{K}_{p+m,1}(D_m,\omega_{D_m})$ is spanned by syzygies of minimal rank $p+m+1$ such that $D_m$ lies in the smooth locus of the syzygy scheme $X_{\alpha}:=\mathrm{Syz}(\alpha)$ and such that we have $\deg(L_{\alpha})=g-p$, $h^0(C,\mu^*L_{\alpha})=2$, where $\mu: C \to D_m$ is the normalization map (note that this implies that all pullbacks of $L_{\alpha}$ along partial-normalizations at any number of nodes have two sections). \smallskip

Set $\ell=g-p-1$, then $D_m$ has arithmetic genus $2\ell$. By Proposition \ref{gen-nondegK3}, 
$$S^* \; : \; \mathrm{H}^0(\mathcal{O}_{\PP}(1)) \to \mathrm{H}^0(S^*\mathcal{O}_{\PP}(1)),$$
with $\PP:=\PP(\mathrm{K}_{\ell-1,1}(D_{m},\omega_{D_m}))$, is injective since $m=g-2p-2 \leq \ell-1$. Note that $W^1_{\ell+1}(D_m)$ is zero-dimensional and reduced, and each element $[A]\in W^1_{\ell+1}(D_m)$ is a line bundle with exactly two sections by Proposition  \ref{gen-node-BN}. The morphism $S$ is the map
	{\small{\begin{align*}
	S \; : \;  X^1_{\ell+1}(D_m) & \to \PP(\mathrm{K}_{\ell-1,1}(D_m,\omega_{D_m})), \\
	[(A,s)] & \mapsto [\mathrm{K}_{\ell-1,1}(D_m,\omega_{D_m}; \mathrm{H}^0(\omega_{D_m} \otimes I_{Z(s)})].
	\end{align*}}}
By Proposition \ref{gen-nondegK3}, the image of $S$ does not lie in any hyperplane and hence the image $S(X^1_{\ell+1}(D_m))$ spans $\PP(\mathrm{K}_{\ell-1,1}(D_m,\omega_{D_m}))$. Thus $\mathrm{K}_{\ell-1,1}(D_m,\omega_{D_m})$ is spanned by syzygies $\alpha$ which lie in $\mathrm{K}_{\ell-1,1}(D_m,\omega_{D_m}; \mathrm{H}^0(\omega_{D_m} \otimes I_{Z(s)})$ for $s \in \mathrm{H}^0(A)$, $A \in W^1_{\ell+1}(D_m)$. By Proposition \ref{last-ruling-prop}, such syzygies $\alpha$ are of minimal rank, $D_m$ lies in the smooth locus of the syzygy scheme $X_{\alpha}$, and the associated ruling is the line bundle $A \in W^1_{\ell+1}(D_m)$.\smallskip

It remains to show that, for any $A \in W^1_{\ell+1}(D_m)$, we have $h^0(C,\mu^*A)=2$. Assume that $h^0(C,\mu^*A)\geq 3$ for some $A \in W^1_{\ell+1}(D_m)$. Notice that, for each $1 \leq i \leq m$, there exists $s_i \in \mu^*\mathrm{H}^0(D_m,A)$ with $$s(x_i)=s(y_i)=0.$$ Consider the space $G^1_{\ell+1}(C)$ of $g^r_d$'s, i.e.\ pairs $V \seq \mathrm{H}^0(C,L)$, where $L \in W^1_{\ell+1}(C)$ and $V$ is a two dimensional, base-point free subspace of the space of global sections of $L$. As $C$ is general, $G^1_{\ell+1}(C)$ is smooth of dimension $\rho(g,1,\ell+1)=m$, \cite[XXI, Prop.\ 6.8]{ACGH1}. If $m=0$ then $D_m=C$ and there is nothing to prove, so assume $m>0$, in which case $G^1_{\ell+1}(C)$ is further irreducible, \cite{fulton-laz-connectedness}. Let $\mathcal{Z} \seq G^1_{\ell+1}(C)$ be the closed subscheme of pairs $V \seq \mathrm{H}^0(C,L)$ with $L \in W^2_{\ell+1}(C)$. Then $\mathcal{Z}$ has codimension at least one in $G^1_{\ell+1}(C)$, \cite[IV, Lemma 3.5]{ACGH1}. On the other hand, if $(x_i,y_i)$ are general for $1 \leq i \leq m$, then the locus $\mathcal{T}$ of pairs $[V \seq \mathrm{H}^0(C,L)] \in \mathcal{Z}$ satisfying the condition 
$$\text{there exist $s_1, \ldots, s_m \in V$ with $s_i(x_i)=s_i(y_i)=0$ for $1 \leq i \leq m$}$$ has codimension at least $m$. Thus $\mathcal{T}=\emptyset$ as required.

\end{proof}

\begin{remark}
Whilst this paper concerns characteristic zero, one may hope progress in characteristic $p$ is possible, as the starting point has recently been proven by Yi Wei \cite{yi-wei}.
\end{remark}

\end{document}